\documentclass[reqno,11pt,a4paper]{amsart}

\usepackage{amsfonts,amsmath,amssymb,amsthm,color}

\usepackage{amssymb}
\usepackage{amsfonts}
\usepackage{amsmath}
\usepackage{amsthm}
\usepackage{nccmath}
\usepackage{mathrsfs}
\usepackage[USenglish]{babel}
\usepackage{esint}
\usepackage{graphicx}
\usepackage[colorlinks=true]{hyperref}
\usepackage{pgfplots}
\usetikzlibrary{patterns}
 \usepackage[hmargin=3cm, vmargin=2.9cm]{geometry}
\hypersetup{
    colorlinks,
    citecolor=red,
    linkcolor=blue,
    urlcolor=black
}

\renewcommand{\(}{\left(}
\renewcommand{\)}{\right)}

\newcommand{\bs}[1]{\boldsymbol{#1}}

\allowdisplaybreaks

\newtheorem{theorem}{Theorem}[section]

\newtheorem{proposition}[theorem]{Proposition}

\newtheorem{remark}[theorem]{Remark}

\numberwithin{equation}{section}

\def\R{{\mathfrak R}}

\def\begeq{\begin{equation}}
\def\endeq{\end{equation}}

\def\R{\Bbb R}

\allowdisplaybreaks

\begin{document}

\title[Segregated solutions for a critical elliptic system]{Segregated solutions for a critical elliptic system   with a small   interspecies repulsive force}


\author{Haixia Chen}
\address{Haixia Chen,
\newline\indent School of Mathematics and Statistics, Central China Normal University
\newline\indent Wuhan 430079, P. R. China.
}
\email{hxchen@mails.ccnu.edu.cn}

\author{Maria Medina}
\address{Maria Medina,
\newline\indent Departamento de Matem\'aticas, Universidad Aut\'onoma de Madrid, Ciudad Universitaria de Cantoblanco  
\newline\indent  28049 Madrid, Spain
}
\email{maria.medina@uam.es}

\author{Angela Pistoia}
\address{Angela Pistoia, 
\newline\indent  Dipartimento di Scienze di Base e Applicate per l’Ingegneria, Sapienza Universita di Roma
\newline\indent  Via Scarpa 16, 00161 Roma, Italy}
\email{angela.pistoia@uniroma1.it}

\begin{abstract}
We consider the elliptic system
$$-\Delta u_i   = u_i^3+\sum\limits_{j=1\atop j\not=i}^{q+1}{ \beta_{ij}}u_i u_j^2\   \hbox{in}\ \mathbb R^4, \ i=1,\dots,q+1.$$
when  
$\alpha:=\beta_{ij}$ and $\beta:=\beta_{i(q+1)}=\beta_{(q+1)j}$ for any $i,j=1,\dots,q.$
If $\beta<0$ and $|\beta|$ is small enough we build solutions  such that  each component $u_{1},\dots,u_q$ blows-up at the vertices of $q$ polygons placed in different {\em great circles}
which are linked to each other, and the last component $u_{q+1}$ looks like the radial positive solution of the single equation.

\end{abstract}

\date\today
 \thanks{ 
H. Chen  has been partially supported by the NSFC grants (No.12071169) and the China Scholarship Council (No. 202006770017). M.~Medina  has been supported by Project PDI2019-110712GB-100, MICINN, Spain, and by Project SI3-PJI-2021-00324, cofunded by Universidad Aut\'onoma de Madrid and Comunidad de Madrid. A. Pistoia has been partially supported by INDAM-GNAMPA project.}

\maketitle

\section{Introduction and main results}
This article is devoted to study the  critical system
\begin{equation}\label{sys0}
-\Delta u_i   = \sum\limits_{j=1}^m{ \beta_{ij}}u_i u_j^2\   \hbox{in}\ \mathbb R^4,\ u_i>0, \ i=1,\dots,m
\end{equation}
where the $\beta_{ij}$'s are real parameters. 
This kind of system   arises from many physical models, for instance in nonlinear optics (e.g. \cite{AA,MCSS,MS}) and in Bose-Einstein condensates for multi-species condensates (e.g. \cite{RCF,TB}). From the physical point of view, $\mu_i:=\beta_{ii}$
  describes  the interaction between particles of the  
component $u_i$, supposed to be positive, while $\beta_{ij}$, $i\not=j,$ describes the interspecies force  between particles   of  the different components $u_i$ and $u_j$, which can be {\em attractive} if $\beta_{ij}>0$ or {\em repulsive} if
  $\beta_{ij}<0$.  
  \\
  
  A more general version of  \eqref{sys0} 
  is the critical system
\begin{equation}\label{sysn}
-\Delta u_i = \sum\limits_{j=1}^m \beta_{ij}u_i ^{\frac{2^*}{2} -1} u_j^{\frac{2^*}{2}}\   \hbox{in}\ \mathbb R^n,\ u_i>0,\ i=1,\dots,m
\end{equation}
where $n\geq3$ and $2^{\ast}:=\frac{2n}{n-2}$ is the critical Sobolev exponent. When $m=1$ system \eqref{sysn} is reduced to the single equation
\begin{equation}\label{equn} -\Delta u= u^{2^*-1}\ \hbox{in}\ \mathbb R^n.\end{equation}
It is well known that all its positive solutions are the so-called {\em bubbles}
 \begin{equation}\label{bubbles}U_{\delta,\xi}(x):= \delta^{-\frac{n-2}{2}}U(\frac{x-\xi}{\delta}),\quad  U(x):=\frac{c_n}{(1+|x|^2)^{\frac{n-2}{2}}},\;c_n>0,\quad x\in\R^n,
 \end{equation} 
 where $\delta>0$ and $\xi\in\mathbb R^n.$ 
 It is useful to remind that \eqref{equn} has also a wide number of changing-sign solutions.
 The first result is due to Ding  \cite{d} who   established the existence of infinitely many
sign-changing solutions for \eqref{equn} which are invariant under the action of a suitable 
 group of conformal transformations of $\mathbb{R}^n$. More recently, Del Pino, Musso, Pacard and Pistoia \cite{dPMPP,dPMPP2}
  constructed sequences of sign-changing solutions with large energy and concentrating along some special
submanifolds of $\mathbb R^n$. In particular, for $n\ge4$ they obtained sequences of solutions
whose energy concentrates along one {\em great circle} or finitely many great circles
which are linked to each other (and they correspond to Hopf links embedded in
$\mathbb S^3\times\{0\}\subset\mathbb S^n$). \\

Concerning system \eqref{sysn}, first of all we observe that it can  have {\em semi-trivial} solutions, i.e. one or more components $u_i$ identically vanish. In this case \eqref{sysn} can be reduced to a system with a less number of components. For example,  if $u$ solves the single equation \eqref{equn}, then
 $u_1=\mu_{1}^{\frac{2-n}{4}}u,$ $ u_2=\dots=u_m=0$  is a semi-trivial  solution to \eqref{sysn}.  
It is also useful to point out that it can have {\em synchronized} solutions, i.e.  the components $u_i=s_i u$ with $s_i>0$ and $u$ solves the single equation \eqref{equn}. In this case system \eqref{sysn} is reduced to an algebraic system
$$s_i=\sum\limits_{j=1}^m \beta_{ij}s_i ^{{2^*\over2} -1}s_j^{{2^*\over2}},\ s_i>0\  i=1,\dots,m.$$
Clearly, one is interested in finding
 \emph{fully nontrivial} (and non-synchronized)
solutions, i.e. when all the components are not identically zero. As far as we know the first result concerning system \eqref{sysn} with only two equations is due to Chen and Zou \cite{cz1,cz2}, who  established the
existence of a positive least energy fully nontrivial solution  in an attractive regime, i.e.  for any $\beta:=\beta_{12}=\beta_{21}>0$ if
$n\geq5$ and for a wide range of $\beta>0$ if $n=4.$ Peng, Peng and Wang \cite{ppw} studied the system for $\mu_{1}=\mu_{2}=1,$
$\beta=1$ and  obtained uniqueness and non-degeneracy results for positive synchronized solutions. Actually, Druet and Hebey  \cite{dh}   proved that  all the positive solutions to \eqref{sysn} in dimension $n=4$ when all the $\beta_{ij}$'s are equal to 1  have to be  synchronized.  In  \cite{ggt1, ggt2} Gladiali, Grossi and Troestler obtained radial and nonradial solutions to \eqref{sysn}  using bifurcation methods.\\

In presence of a  repulsive regime, there is a strong connection between positive solutions to the system \eqref{sysn} and changing-sign solutions to the single equation \eqref{equn}, as pointed out in a series of papers by Terracini and her collaborators (see \cite{11,29,37,38,39}) because a complete {\em segregation phenomena} between components occurs as the interaction forces $\beta_{ij}\to-\infty$. So it is natural to ask if whenever
there exists a changing-sign solution to the single equation \eqref{equn}  it is   possible to find a positive solution to the system \eqref{sysn} whose components resemble its positive and   negative part. 

This is true  when we consider the changing-sign solutions found by Ding \cite{d}. Indeed 
Clapp and Pistoia \cite{cp} (when $m=2$) and Clapp and Szulkin \cite{cs} (when $m\ge3$), using a similar variational approach, found   fully 
non-trivial  positive solutions to \eqref{sysn}  in a fully repulsive regime, i.e.   $\beta_{ij}<0$. This is also true when we consider the changing-sign solutions found by Del Pino, Musso, Pacard and Pistoia \cite{dPMPP} in dimension $n=3.$ 
Indeed Guo, Li and Wei  \cite{glw}, using a Ljapunov-Schmidt reduction procedure,  built solutions to system \eqref{sysn} (with $m=2$) when $\beta:=\beta_{12}=\beta_{21}<0,$ whose  first component looks like the radial positive solution $U$ in \eqref{bubbles} and the second component resembles the sum of $k$ negative bubbles $U_{\delta,\xi_i}$ whose peaks $\xi_i$ are arranged on a regular polygon with radius around 1. 

Now, since in \cite{dPMPP,dPMPP2} the authors built a wide number of sign-changing solutions with different profiles in dimension $n\ge4$, it is natural to  ask if there is also a connection  among them and positive solutions to the system \eqref{sysn} with $m\ge2$ and $n\ge4.$
This is not surely possible using Ljapunov-Schmidt-type techniques if $n\ge5$, since the coupling term $u_i^{{2^*\over2} -1}u_j^{{2^*\over2}}$ has 
  sub-linear growth (i.e. ${2^*\over2} -1<1$ if $n\ge5$) and the reduction process does not work. 
In the present paper, we will focus on the case $n=4$ where the coupling term $u_iu_j^2$ has linear growth and we prove the existence of an arbitrary large number of {\em segregated} solutions whose components resemble the profile of the solution found in  \cite{dPMPP,dPMPP2} in the presence of  a small interspecies repulsive force. \\
  
  From now on, we will consider the system \eqref{sys0} with $\mu_i=1$
\begin{equation}\label{sys03}
-\Delta u_i=u_i^3+\sum\limits_{j\neq i}{ \beta_{ij}}u_i u_j^2\quad  \hbox{in}\ \mathbb R^4,\ i=1,\dots,m,
\end{equation}
where the $\beta_{ij}$'s satisfy
\begin{equation}\label{beta}\left\{
\begin{aligned}&\beta_{ij}=\beta_{\kappa\ell}\ \hbox{for any}\ i,j,\kappa,\ell=1,\dots,m-1,\ i\not=j\ \hbox{and}\ \kappa\not=\ell\\
&\beta_{im}=\beta_{mj}=:\beta<0\ \hbox{if}\ i,j=1,\dots,m-1.\end{aligned}\right.
\end{equation}
If $m=2$ assumption \eqref{beta} reduces to $\beta_{12}=\beta_{21}=:\beta<0.$

We will construct the solutions of \eqref{sys03} by means of a Ljapunov-Schmidt reduction method. Roughly speaking, for every component we will find $u_i$ given as a small perturbation of some explicit function $u_i^*$, let us say
\begin{equation}\label{approx}
u_i=u_i^*+\phi_i,\quad i=1,\ldots,m,
\end{equation}
so the first step to solve the system will be choosing appropriate approximations $u_i^*$. Let us introduce how we will face it. Given a parameter $\delta>0$ and a point $\xi\in \R^4$, consider the functions
 \begin{equation}\label{bubble}U_{\delta,\xi}(x):=\frac1\delta U\(x-\xi\over\delta\),\quad  U(x):={c_4\over1+|x|^2},\;c_4:=2\sqrt2,\quad x\in\R^4,
 \end{equation} 
 which are just \eqref{bubbles} in the case $n=4$. Hence $U$ and $U_{\delta,\xi}$, for every $\delta>0$, $\xi\in\R^4$, conform all the possible positive solutions to the single critical equation
 \begin{equation}\label{e}
 -\Delta U =U^3\ \hbox{in}\ \mathbb R^4.
 \end{equation}
The idea will be to construct our approximation by {\it gluing} an arbitrarily large number of copies of these functions, centered at certain precise points of $\R^4$. We will analyze the solvability of \eqref{sys03} in two phases, first the two components case and then the general case $m\geq 3$.

Let us start by assuming $m=2$. Inspired by \cite{dPMPP}, we will use the function $U$ as approximation for the first component, and for the second we will choose 
$$u_1^*(x):=\sum_{i=1}^k U_{\delta,\xi_i}(x),$$
where $k\geq 2$, $\delta>0$, and the points $\xi_i$ are arranged at the vertices of a regular polygon. More precisely,
 \begin{equation}\label{bas03intro}
\xi_i:=\rho\(\cos\frac{2\pi(i-1)}k,\sin\frac{2\pi(i-1)}k ,0,0\right),\quad  \delta^2+\rho^2=1,\quad  i=1,...,k.
\end{equation}
Notice that, due to the non-linearity of the equation, $u_2^*$ is not a solution of \eqref{e}. The precise choice of the location of the points $\xi_i$ allows us to prove that the approximation keeps several symmetries (see \eqref{sim1}-\eqref{kel}) which will be crucial in the solvability theory that we will develop to find the perturbations $\phi_i$, according to the notation in \eqref{approx}  (see Proposition \ref{prop1}). Given $\phi \in\mathcal{D}^{1,2}(\R^4)$, we consider the norm
$ \|\phi\|:=\left(\int_{\R^4}|\nabla \phi|^2\right)^{1/2}.$
We will prove the following result:
 
\begin{theorem}\label{thm01}
Let $m=2$ and assume \eqref{beta}. For every fixed $k\geq 2$, there exists $\beta_0<0$ such that, for each $\beta\in [\beta_0,0)$, the system \eqref{sys03} has a solution $\bs u=(u_1, u_2)$ with the form
\begin{equation*}
u_1=\sum_{i=1}^k U_{\delta,\xi_i}+\psi,\quad u_2=U+\phi, \quad   \text{   in   }  \mathbb{R}^{4},
\end{equation*}
with $U, U_{\delta,\xi_i}$ defined in  \eqref{bubble}, $\xi_i$ in \eqref{bas03intro}, and  $\psi,\phi\in \mathcal{D}^{1,2}(\R^4)$. Furthermore,
$$\|\psi\|,\|\phi\|\to 0, \quad \delta=\mathcal O\bigg(e^{-c\frac{1}{| \beta|}}\bigg),\ \hbox{for some constant}\ c>0,$$
 as $|\beta|\to 0$.
\end{theorem}

The case $m\geq 3$ is much more involved. Here one polygon is not enough since we need $m$ different approximations. We do it following an idea of \cite{dPMPP2}, where a sign-changing solution to \eqref{e} is found setting copies of $U_{\delta,\xi}$ at the vertices of $q$ polygons placed in different great circles, with $q$ an arbitrary number. We will use each of these polygons as approximation for each component. More precisely, we will consider the functions
\begin{equation}\label{approx3}
u_r^*(x)=\sum\limits_{\ell=1}^kU_{\delta, \tilde \xi_\ell^r}(x),\quad  r=1,\dots, m-1, \quad u_{m}^*(x)=U(x),\end{equation}
where, denoting $q:=m-1$,
\begin{equation}\label{tildexi}
\begin{aligned}\tilde \xi_\ell^r:=\frac{\rho}{\sqrt2} \Big(&\cos\Big(-\frac{(r-1)\pi}q+ \frac{2\pi(\ell-1)} k\Big) ,\sin\Big(-\frac{(r-1)\pi}q+ \frac{2\pi(\ell-1)} k\Big),\\ &\cos\Big(\frac{(r-1)\pi}q+ \frac{2\pi(\ell-1)} k\Big) ,\sin\Big(\frac{(r-1)\pi}q+ \frac{2\pi(\ell-1)} k\Big)\Big),\end{aligned}\end{equation}
and $\delta^2+\rho^2=1$. Notice that, defining the linear transformation 
$$ {\mathscr T}_\theta:=\(\begin{matrix}\cos\theta&-\sin\theta&0&0\\
\sin\theta&\cos\theta&0&0\\
0&0&\cos\theta&\sin\theta\\
0&0&-\sin\theta&\cos\theta \\
 \end{matrix}\),$$
 and considering the initial configuration
 \begin{equation*}  \tilde\xi_\ell:= \frac{\rho}{\sqrt2}\Big( \cos\frac{2\pi(\ell-1)} k, \sin\frac{2\pi(\ell-1)} k, \cos\frac{2\pi(\ell-1)} k, \sin\frac{2\pi(\ell-1)} k\Big),\quad \ell=1,\dots,k,
\end{equation*}
then we have that $\tilde \xi_\ell^r={\mathscr T}^{-1}_{\frac{(r-1)\pi}{q}}\tilde\xi_\ell$ and therefore
$$u_r^*(x)=u_{r-1}^*({\mathscr T}_{\frac{\pi}{q}}x),\quad r=2,\ldots,m-1.$$
This subtle symmetry in the construction will allow us, under some restrictions on the coefficients $\beta_{ij}$, to reduce the general system to a new one with only two equations. Thus the first step in the case $m\geq 3$ will be to exploit the structure of \eqref{approx3} to reduce the system to \eqref{2222}, and the second to solve the new system. At first sight this may look very similar to \eqref{sys03} with $m=2$. However, the new terms that appear from the interaction between the $m$ components do not inherit the symmetries \eqref{sim1}-\eqref{kel} required to prove Theorem \ref{thm01}, so a new analysis is needed. A fundamental issue here will be to identify the special invariances of the construction, much less intuitive (see \eqref{24}-\eqref{e63}), and the functional spaces where we will be able to develop the appropriate linear theory (see Proposition \ref{prop4.1}).

\begin{theorem}\label{thm1.3}
Let $m\geq 3$ and assume \eqref{beta}.
For every fixed $k\geq 2$ even, there exists $\beta_0<0$ such that, for each $\beta\in [\beta_0,0)$, the system \eqref{sys03} has a solution $\bs u=(u_1,..., u_m)$ with the form
\begin{equation*}
u_r=\sum\limits_{\ell=1}^kU_{\delta, \tilde \xi_\ell^r}+\psi_r,\quad  r=1,\dots, m-1, \quad u_{m}=U+\phi, \text{~   in~   }  \mathbb{R}^{4},
\end{equation*}
with $U, U_{\delta,\tilde \xi_\ell^r}$ defined in  \eqref{bubble}, $\tilde \xi_\ell^r$ in \eqref{tildexi}, and  $\psi_r,\phi\in \mathcal{D}^{1,2}(\R^4)$. Furthermore, for every $r=1,\ldots, m-1$,
$$\|\psi_r\|,\|\phi\|\to 0, \quad \delta=\mathcal O\bigg(e^{-c\frac{1}{| \beta|}}\bigg),\ \hbox{for some constant}\ c>0,$$
 as $|\beta|\to 0$.
\end{theorem} 
\begin{remark}
Regarding the positivity, it is natural to think that in Theorems \ref{thm01} and \ref{thm1.3} we have a positive solution without any additional assumptions on $\beta_{ij}$. Indeed, any component $u_{i}$ is a superposition of positive function and small perturbation term. If the positive part of $\beta_{ij}$ is small for every $i \neq j$, then a short rigorous proof of the positivity can be given arguing as in \cite{PisTav}. If on the other hand some $\beta_{ij}$ is allowed to be large, such proof does not work and one is forced to approach the problem with finer (and much longer) techniques, such as careful $L^\infty$-estimates on the error  terms $\phi$ and $\psi$. We decided not to insist on this point for the sake of brevity.\end{remark}
\begin{remark} We strongly believe that results stated in Theorems \ref{thm01} and \ref{thm1.3} can be rephrased taking $k$ as large parameter and $\beta<0$ fixed,  in the spirit of Del Pino, Musso, Pacard and Pistoia's results.
Actually, the main obstacle in proving this result is the invertibility of the  linear operator  (see \eqref{s2}) in the standard space $\mathcal D^{1,2}(\mathbb R^4)\times \mathcal D^{1,2}(\mathbb R^4).$ It would be necessary to choose some   $L^p$-weighted spaces where
the linear theory works, but (as far as we can say) it is not clear at all.
\end{remark}
Besides this introduction, the article is organized in two main sections, containing the case of two components ($m=2$, in Section \ref{sec2}) and the general one ($m\geq 3$, in Section \ref{sec4}) respectively. In both parts the same structure is followed: analysis of the general setting of the problem, with special attention to the invariances of the constructions and the functional spaces;  invertibility of the linear operator obtained in the problem satisfied by the perturbations;  size of the error of the approximation; solvability of a projected non-linear problem and, finally, the reduction argument. 

\section{The case $m=2$}\label{sec2}
Assume $k\geq 2$ and $\beta<0$ along the whole section. We are interested in finding solutions to the system 
\begin{equation}\label{sys2}
\left\{\begin{aligned}&-\Delta u=u^3+\beta uv^2\quad \hbox{in}\ \mathbb R^4,\\
&-\Delta v=v^3+\beta vu^2\quad \hbox{in}\ \mathbb R^4,\\
\end{aligned}\right.
\end{equation}
of the form
\begin{equation}\label{ans}
u=U+\phi,\quad  v=\underbrace{\sum\limits_{i=1}^kU_{\delta, \xi_i}}_{V} +\psi 
\end{equation}
 where  $k$ is fixed, $\phi$ and $\psi$ are small functions to be found, and $U$ and $U_{\delta, \xi_i}$ are given by \eqref{bubble} with
 \begin{equation}\label{bas03}
\xi_i:=\rho\(\cos\frac{2\pi(i-1)}k,\sin\frac{2\pi(i-1)}k ,0,0\right),\quad  \delta^2+\rho^2=1,\quad  i=1,...,k.
\end{equation}
 
\subsection{Setting of the problem}
Consider the space $ \mathcal D^{1,2}(\mathbb R^4)$, which is a Hilbert space 
 equipped with the scalar product and the induced norm
$$\langle\phi,\psi\rangle:=\int\limits_{\mathbb R^4}\nabla \phi\nabla \psi\qquad  \|\phi\|:=\(\int\limits_{\mathbb R^4}|\nabla \phi|^2\)^\frac12.$$
It is known that the embedding $ \mathcal D^{1,2}(\mathbb R^4)\hookrightarrow L^4\(\mathbb R^4\)$ is continuous. We define, via the Riesz representation theorem, the operator $(-\Delta)^{-1}: L^\frac43\(\mathbb R^4\)\to \mathcal D^{1,2}(\mathbb R^4)$ as
$$(-\Delta)^{-1}(f)=u\ \Leftrightarrow\ -\Delta u=f\ \hbox{in}\ \mathbb R^4.$$
It is immediate to check that there exists $c>0$ such that
$$
\|(-\Delta)^{-1}(f)\|\le c\|f\|_{L^\frac43(\mathbb R^4)}\ \hbox{for any}\ f\in L^\frac43\(\mathbb R^4\),$$
and then the problem \eqref{sys2} can be rephrased as
\begin{equation}\label{sys21}
\left\{\begin{aligned}&  u=(-\Delta)^{-1}\(u^3+\beta uv^2\),\\
&  v=(-\Delta)^{-1}\(v^3+\beta vu^2\).\\
\end{aligned}\right.
\end{equation}
Given a function $\phi\in \mathcal D^{1,2}(\mathbb R^4)$, let us consider the following invariances:
\smallskip

$\bullet$ Evenness with respect to the $x_2,x_3,x_4$ variables, i.e.,
\begin{equation}\label{sim1} \phi(x_1,x_2,x_3,x_4)=\phi(x_1,-x_2,x_3,x_4)=\phi(x_1,x_2,-x_3,x_4)=\phi(x_1,x_2,x_3,-x_4).\end{equation}

$\bullet$ Invariance under rotation of $\frac{2\pi}k$ in the $x_1,x_2$-variables, i.e.,
\begin{equation}\label{sim2}\phi( \Theta_k(x_1,x_2),x_3,x_4)=\phi(x_1,x_2,x_3,x_4),  \quad \Theta_k:=\(\begin{matrix}\cos\frac{2\pi}k &-\sin\frac{2\pi}k\\
\sin\frac{2\pi}k &\cos\frac{2\pi}k\\  \end{matrix}\).
\end{equation}

$\bullet$ Invariance under Kelvin transform, i.e.,
\begin{equation}\label{kel}\phi(x)=\frac1{|x|^2}\phi\(\frac{x}{|x|^2}\).\end{equation}
We define the associated space
 $$ X:=\{\phi\in \mathcal D^{1,2}(\mathbb R^4)\ : \ \phi\ \hbox{satisfies \eqref{sim1}, \eqref{sim2} and \eqref{kel}}\}.$$
Indeed, attending to the structure of \eqref{sys21}, we will look for solutions $(u,v)$ in the space $X\times X$. Notice that if $u,v\in X,$ then also the right hand side of \eqref{sys21} belongs to  the same space $X\times X.$ In particular, $(-\Delta)^{-1}(f)$ is Kelvin invariant if $f$ satisfies
\begin{equation*}{f}(x)=\frac1{|x|^6}{ f}\(x\over |x|^2\).\end{equation*}
Let us consider the linearization of the Yamabe equation around the function $U$, 
 \begin{equation}\label{lin01}
 -\Delta \phi=3U^2\phi\ \hbox{in}\ \mathbb R^4.
 \end{equation}
It is known that the set of  solutions in $\mathcal D^{1,2}(\mathbb R^4)$ is a $5$-dimensional space, whose generators are
 \begin{equation}\label{lin02}
Z^0(x):={1-|x|^2\over(1+|x|^2)^2}\quad  \hbox{and}\quad   Z^j(x):={x_j\over (1+|x|^2)^2},\quad j=1,\dots,4.
 \end{equation}
Given $\delta>0$ and $\xi_i$ given in \eqref{bas03}, we can define
\begin{equation}\label{z0}
Z_{\delta,\xi_j}^i(x):=\frac1\delta Z^i\(x-\xi_j\over\delta\),\quad i=0,1,\dots,4,\quad j=1,\dots,k,
\end{equation}
and
$$Z(x):=\sum\limits_{i=1}^kZ_{\delta,\xi_i}^0(x).$$
We introduce the spaces
$$K:=X\cap \mathtt{span} \left\{Z\right\},\quad K^\perp:=\left\{\phi\in X:\, \langle\phi,Z\rangle=0\right\},$$
and the orthogonal projections 
$$\Pi:X\times X\to X\times K\quad \mbox{ and }\quad \Pi^\perp:X\times X\to X\times K^\perp.$$
Notice that the function $Z$ belongs to $X$. If $(u,v)$ are given by \eqref{ans}, we can rewrite the system
\eqref{sys21}   in terms of $(\phi,\psi)$ as 

\begin{equation}\label{rid1}
\Pi \left\{\bs{\mathcal L}^* \(\phi,\psi\)-\bs{\mathcal E}^*-\bs{\mathcal N}^*\(\phi,\psi\)\right\}=0,
\end{equation}
\begin{equation}\label{rid2}
\Pi ^\perp\left\{\bs{\mathcal L}^* \(\phi,\psi\)-\bs{\mathcal E}^*-\bs{\mathcal N}^*\(\phi,\psi\)\right\}=0,
\end{equation}
where
the linear operator $\bs{\mathcal L}^*=(\mathcal L_1^*,\mathcal L_2^*)$ is defined by
$$\left\{\begin{aligned}&\mathcal L_1^*(\phi,\psi):=\phi -(-\Delta)^{-1}
\(3U^2\phi\)\\
&\mathcal L_2^*(\phi,\psi):=  \psi -(-\Delta)^{-1}\(3V^2\psi \),\\
\end{aligned}\right.$$
the  error term $\bs{\mathcal E}^*=(\mathcal E_1^*,\mathcal E_2^*)$ by
\begin{equation}\label{errorDef}\left\{\begin{aligned}&\mathcal E_1^*:=(-\Delta)^{-1}\(\beta UV^2\),\\
 &\mathcal E_2^*:=(-\Delta)^{-1}\(V^3+\Delta V+\beta U^2 V\),
 \end{aligned}\right.\end{equation}
and the quadratic term $\bs{\mathcal N}^* =(\mathcal N_1^* ,\mathcal N_2^*)$ by
\begin{equation}\label{quad2}\left\{\begin{aligned}&\mathcal N_1^*\(\phi,\psi\):=(-\Delta)^{-1}\(\phi ^3+3U\phi^2+\beta(\phi\psi^2+2V\phi\psi+U\psi^2)+\beta V^2\phi+2\beta UV\psi\),\\
&\mathcal N_2^*\(\phi,\psi\):=(-\Delta)^{-1}\(\psi^3+3V\psi^2+\beta(\phi^2\psi +2U\phi\psi+V\phi^2)+\beta U^2\psi+2\beta UV\phi\). \end{aligned}\right.\end{equation}
We will start by solving \eqref{rid2}. With this purpose, let us denote 
$$\bs{\mathcal L}:=\Pi^{\perp}\bs{\mathcal L}^*,\quad  \bs{\mathcal E}:=\Pi^{\perp}\bs{\mathcal E}^*, \quad\bs{\mathcal N}:=\Pi^{\perp}\bs{\mathcal N}^*.$$

 \subsection{The invertibility of the linear operator $\bs{\mathcal L}$ }\label{s2} 

\begin{proposition}\label{prop1}
There exist $c>0$ and  $\beta_0<0$ such that for each $\beta \in [\beta_0, 0)$ and $\delta\in\big(0,e^{-\frac{1}{\sqrt{|\beta|}}}\big),$  
\begin{equation}\label{linear}\|\bs{\mathcal L}(\phi,\psi)\|\ge c\|(\phi,\psi)\|, \quad \forall\, (\phi, \psi)\in X\times K^{\perp}.\end{equation}
In particular, the inverse operator $\bs{\mathcal L}^{-1}:X\times K^\perp\to X\times K^\perp$ exists and is continuous.
\end{proposition}
\begin{proof}
Assume first that \eqref{linear} holds. The operator  $\bs{\mathcal{L}}$ can be seen as a compact perturbation of the identity map in $X\times K^{\perp}$. Indeed,
$$\bs{\mathcal{L}}(\phi,\psi)=(\phi,\psi)-\underbrace{\Pi^{\perp}((-\Delta)^{-1}
(3U^2\phi),(-\Delta)^{-1}(3V^2\psi ))}_{=:\mathscr  K(\phi,\psi)}$$
where $\mathscr  K$ is compact since $U^2,V^2\in L^2(\R^4)$ for each fixed $k$ (see \cite[Lemma 2.2]{BS}).
Therefore, invertibility follows by Fredholm's alternative theorem
and \eqref{linear}.

Let us prove \eqref{linear}. Assume by contradiction that there exist $\beta_n\to 0$ and $(\phi_n,\psi_n)\in X\times K^\perp$ such that
\begin{equation}\label{ipo}\bs{\mathcal{L}}(\phi_n,\psi_n)=(f_n,g_n)\in X\times K^\perp,\quad \|\phi_n\|+\|\psi_n\| =1, \quad \|f_n\|+\|g_n\|=o(1).\end{equation}
More precisely, there exists $t_n$ such that the functions $\phi_n$ and   $\psi_n$   solve
\begin{equation}\label{sis0}\left\{\begin{aligned}& -\Delta \phi_{n}-3U^2\phi_{n} =-\Delta f_n    \quad \hbox{in}\ \mathbb R^4,\\
& -\Delta \psi_n-3V_n^2\psi_{n}=-\Delta(g_n   +t_nZ_n)    \quad \hbox{in}\ \mathbb R^4,
\end{aligned}\right.\end{equation}
where  
$$V_n:= \sum\limits_{i=1}^{k}U_{in},\quad  U_{in}:=U_{\delta_n,\xi_{in}},\quad  Z_n:=\sum\limits_{i=1}^{k}Z_{in},\quad Z_{in}:=Z^0_{\delta_n,\xi_{in}},$$
with $$\xi_{in}:=\rho_n\Big(\cos\frac{2\pi(i-1)}{k},\sin\frac{2\pi(i-1)}{k} ,0,0\Big),\  \quad \rho_n^2+\delta_n^2=1, \quad {\delta_n\in (0, e^{-\frac{1}{\sqrt{|\beta_n|}}})}.$$\\
\textbf{Step 1.}  Let us prove that (up to a subsequence), as $n\to+\infty$,
\begin{equation*}\phi_{n} \to0\ \hbox{weakly in}\ \mathcal D^{1,2}(\mathbb R^4)\ \hbox{and}\   \hbox{strongly in}\ L^p_{loc}(\mathbb R^4)\ \hbox{for any}\ p\in [2,4).\end{equation*}
By \eqref{ipo} we deduce that, up to a subsequence, there exists $\phi\in  X$ such that  
\begin{equation*}
\phi_{n}\to\phi\ \hbox{weakly in}\ \mathcal D^{1,2}(\mathbb R^4) \ \hbox{and strongly in}\ L^p_{loc}(\mathbb R^4)\ \hbox{for any}\ p\in [2,4), \end{equation*}
 and satisfies
 $$-\Delta \phi-3 U^2 \phi=0 \text{~in~}\R^4.$$
By \eqref{lin02}, this implies $\phi\in\mathtt{span}\{Z^j\}$, $j=0,1,\ldots, 4$. Thus, if we prove
\begin{equation}\label{key0} 
\langle \phi,Z^j\rangle=0\ \hbox{for any}\ \phi\in X\ \hbox{and}\  j=0,1,\dots,4,\end{equation}
we conclude that necessarily $\phi=0$, and the Step 1 is concluded.

Indeed, by \eqref{lin01}
$$I_j:=\langle \phi,Z^j\rangle=\int\limits_{\mathbb R^4}\nabla \phi \nabla Z^j=\int\limits_{\mathbb R^4}3U^2Z^j\phi,$$
and \eqref{sim1} immediately implies $I_j=0$ if $j=2,3,4$. Furthermore, by \eqref{sim2},
$$ \int\limits_{\mathbb R^4}  {x_1\over(1+|x|^2)^4}\phi(x)dx  
= \left(\cos{2\pi\over k}\right)\int\limits_{\mathbb R^4} {y_1\over(1+|y|^2)^4} \phi (y)dy
 $$
and thus $I_{1}=0$ since $k\geq 2$.
Finally, from \eqref{kel},
$$I_{0} =\int\limits_{\mathbb R^4} {1-|x|^2\over(1+|x|^2)^4}\phi(x)dx =
 \int\limits_{\mathbb R^4}  {|z|^2-1\over |z|^2}{|z|^8\over(1+|z|^2)^4}{1\over |z|^8}\phi\bigg(\frac{z}{|z|^2}\bigg)dz =-I_{0},$$
and \eqref{key0} holds.\\
\textbf{Step 2.}
Denote $ \tilde{\psi}_{n}(y):=\delta_n\psi_n(\delta_n y+\xi_{1n }).$
We will prove that, up to a subsequence, 
\begin{equation}\label{step2}  \tilde \psi_{n}\to 0 \hbox{ weakly in }  \mathcal D^{1,2}(\mathbb R^4)\ \hbox{and}\  \hbox{ strongly in } L^p_{loc}(\mathbb R^4)\ \hbox{for any}\ p\in [2,4). \end{equation}
Indeed, by \eqref{ipo} there exists  $\psi\in \mathcal D^{1,2}(\mathbb R^4)$  such that
\begin{equation*}\tilde \psi_{n}\to \psi \hbox{ weakly in }  \mathcal D^{1,2}(\mathbb R^4) \ \hbox{and} \ \hbox{ strongly in } L^p_{loc}(\mathbb R^4)\ \hbox{for any}\ p\in [2,4). \end{equation*}
To see the equation that $\psi$ satisfies, let us prove first that
\begin{equation}\label{tn}t_n\to 0 \text{~and~} \|Z_n\|^2=A_1k+o(1) \text{~as~} n\to +\infty.\end{equation}
Suppose $i\neq j$ and consider $\vartheta>0$ small enough so that $B(\xi_{in}, \vartheta) \cap B(\xi_{jn}, \vartheta)=\emptyset$. Thus, there exists $C>0$ such that 
$$|Z_{in}|\leq CU_{in},\quad  |Z_n|\leq CV_n, \quad U_{jn}\leq C\delta_n \quad \text{~in~} B(\xi_{in}, \vartheta) \quad \text{~if~} i\neq j.$$
Furthermore,
$$ 
\int\limits_{B(\xi_{in}, \vartheta)}U^3_{in} U_{jn}=\mathcal O(\delta_n^2), \quad \int\limits_{B(\xi_{jn}, \vartheta)}U^3_{in} U_{jn} =\mathcal O(\delta^4_n), \quad 
\int\limits_{\R^4\setminus (B(\xi_{in}, \vartheta)\cup B(\xi_{jn}, \vartheta)) }U^3_{in}U_{jn} =\mathcal O(\delta_n^4).
 $$
Putting this information together we conclude that
\begin{equation*}\begin{aligned}\int\limits_{\R^4}|\nabla Z_n|^2=&\sum\limits_{i=1}^{k}\int\limits_{\R^4}|\nabla Z^0_{\delta_n, \xi_{in}}|^2+2\sum\limits_{i\neq j}\int\limits_{\R^4}\nabla Z^0_{\delta_n, \xi_{in}}\nabla Z^0_{\delta_n, \xi_{jn}}\\=&A_1k+2\sum\limits_{i\neq j}\int\limits_{\R^4}3U^2_{\delta_n, \xi_{in}}Z^0_{\delta_n, \xi_{in}}Z_{\delta_n, \xi_{jn}}^0\\
= &A_1k+o(1).\end{aligned}\end{equation*}
Testing in the second equation of \eqref{sis0} with $Z_n$ and using the fact that $g_n\in K^{\perp}$ we get
\begin{equation}\label{eqtn}
\begin{aligned}0&=\int\limits_{\mathbb R^4}
\nabla \psi_n\nabla Z_n-3V_n^2\psi_{n}Z_n -\nabla (g_n   +t_nZ_n)  \nabla Z_n\\
&=\int\limits_{\mathbb R^4}3\Big(\sum\limits_{i=1}^{k}U_{ {in}}^2Z_{ {in}}-V_n^2Z_n\Big)\psi_n
-t_n\int\limits_{\R^4}|\nabla Z_n|^2.
\end{aligned}
\end{equation}
Furthermore, since 
$$\|U_{in}^2U_{jn}\|_{L^{\frac43}(\R^4)}\leq C\delta^2_n,\quad 
\|U_{in}U_{jn}U_{ln}\|_{L^{\frac43}(\R^4)}\leq C\delta_n^3,
$$
we see that
\begin{equation}\label{l3}\begin{split}
\Big|\int\limits_{\mathbb R^4}3\Big(\sum\limits_{i=1}^{k}U_{in}^2Z_{in}&-V_n^2Z_n\Big)\psi_n\Big|=\Big|\int\limits_{\R^4}3\Big[\sum\limits_{i=1}^{k}U_{in}^2Z_{in}-\Big(\sum\limits_{i=1}^{k}U_{in}^2+2\sum_{i\neq j}U_{in}U_{jn}\Big)\sum\limits_{l=1}^{k}Z_{ln}\Big]\psi_n\Big|\\
&=\Big|\int\limits_{\R^4}3\Big(\sum\limits_{i\neq j}U_{in}^2Z_{jn}+2\sum\limits_{l=1}^{k}\sum\limits_{i\neq j}U_{in}U_{jn}Z_{ln}\Big)\psi_n\Big|\\
&\leq C\Big(\|\sum\limits_{i\neq j}U_{in}^2U_{jn}\|_{L^{\frac43}(\R^4)}+\|\sum\limits_{l=1}^{k}\sum\limits_{i\neq j}U_{in}U_{jn}U_{ln}\|_{L^{\frac43}(\R^4)}\Big)\|\psi_n\|\\
&\leq C\Big(\sum\limits_{i\neq j}\|U_{in}^2U_{jn}\|_{L^{\frac43}(\R^4)} +\sum\limits_{i\neq j\neq l}\|U_{in}U_{jn}U_{ln}\|_{L^{\frac43}(\R^4)}\Big)\|\psi_n\|\\
&\leq C\delta^2_n,
\end{split}\end{equation}
Passing to the limit in \eqref{eqtn} we conclude that $t_n \to 0$ as $n\to +\infty$, and \eqref{tn} follows. Futhermore, it is easy to check that $\psi\in \mathtt{span}\{Z^j\}$, $j=0,\ldots, 4,$ since it satisfies
$$-\Delta \psi-3U^2\psi=0 \text{~in~} \R^4.$$
Notice that, by definition of $\xi_{1n}$, $\tilde\psi_n $ is even with respect to $x_2,x_3,x_4$, and thus $\langle \psi,Z^j\rangle=0$ for $j=2,3,4$. Since $\psi_n\in K^\perp$,
$$0=\int\limits_{\mathbb R^4} \nabla \psi_n\nabla Z_n=3\int\limits_{\mathbb R^4} \(\sum\limits_{i=1}^{k}U_{in}^2Z_{in}\) \psi_n=3 k
\int\limits_{\R^4}  U_{1n}^2Z_{1n}\psi_n
 \to 3k\int\limits_{\R^4}U^2Z^0\psi,$$
and thus $\langle \psi,Z^0\rangle=0$. Finally, since $\delta_n^2+\rho_n^2=1$ and $\psi_n$ is Kelvin invariant (see \eqref{kel}), we obtain
 \begin{equation}\label{Kelv13}\begin{aligned}\int\limits_{ \R^4} \tilde \psi_n(y){y_1\over(1+|y|^2)^4}dy&=\delta_n^{4}\int\limits_{ \R^4}\psi_n(x){x_1-\rho_n\over(\delta_n^2+|x-\xi_{1n}|^2)^4}dx\\
&=\delta_n^{4}\int\limits_{ \R^4}\psi_n(z){z_1-\rho_n|z|^2\over(\delta_n^2+|z-\xi_{1n}|^2)^4}dz\\
&= \int\limits_{ \R^4} \psi_n(\delta_n y+\xi_{1n}){\delta_n y_1+\rho_n(1-|\delta_n y+\xi_{1n}|^2)\over(1+|y|^2)^4}dy\\
&= \int\limits_{ \R^4}  \tilde\psi_n(y){  y_1 \over(1+|y|^2)^4}dy +\frac{\rho_n}{ \delta_n}\int\limits_{ \R^4} \tilde \psi_n(y){ 1-|\delta_n y+\xi_{1n}|^2\over(1+|y|^2)^4}dy,
\end{aligned}
\end{equation}
and then
\begin{equation*}\begin{aligned}0&=\frac{\rho_n}{ \delta_n}\int\limits_{  \R^4} \tilde  \psi_n(y){ 1-|\delta_n y+\xi_{1n}|^2\over(1+|y|^2)^4}dy\\
&= \delta_n\rho_n \int\limits_{  \R^4}  \tilde \psi_n(y) {1-  |y|^2\over(1+|y|^2)^4}dy-2\rho_n^2 \int\limits_{ \R^4}  \tilde \psi_n(y){  y_1 \over(1+|y|^2)^4}dy.
\end{aligned}\end{equation*}
Passing to the limit we deduce 
$$\int\limits_{\mathbb R^4}  \psi(y){  y_1 \over(1+|y|^2)^4}dy=0,$$
and hence $ \langle \psi, Z^1\rangle=0$. Therefore $\psi=0$ and \eqref{step2} follows.
\\
\textbf{Step 3:}  We will prove that 
$$\int\limits_{ \R^4}|\nabla\phi_{n}|^2+|\nabla \psi_n|^2\to 0\quad  \mbox{as} \; n\to \infty,$$ reaching a contradiction with \eqref{ipo}. 

Testing the first equation in \eqref{sis0} with $\phi_n$, the second with $\psi_n$ and summing up we get
\begin{align}1=&\Big(\int\limits_{\mathbb R^4}|\nabla \phi_{n}|^2+\int\limits_{ \R^4}|\nabla \psi_n|^2\Big)\nonumber\\
=&\int\limits_{\R^4}3U^2\phi_n^2+ \int\limits_{\R^4}\nabla f_n\nabla\phi_{n}+\int\limits_{\R^4}3V_n^2\psi_n^2  +\int\limits_{\R^4}\nabla(g_n+t_nZ_n)\nabla\psi_n, \label{con02} \end{align}
By \eqref{ipo} and \eqref{tn} it can be checked that, as $n\to \infty,$
$$ \int\limits_{\R^4}\nabla f_n\nabla \phi_n \to 0,\quad  \int\limits_{\R^4}\nabla  g_n \nabla \psi_n \to 0,\quad  t_n\int\limits_{\R^4}\nabla Z_n \nabla \psi_n  \to 0. $$
Likewise, from Step 1 we obtain $$\int_{\R^4} 3U^2\phi_n^2 \to 0.$$
To estimate the term $\int\limits_{\R^4}3V_n^2\psi_n^2$ let us introduce the cone
$$\Sigma_k:=\left\{y=(r\cos\theta,r\sin\theta,y_3,y_4) :\ |\theta|\le \frac\pi {k}, r\geqslant 0\right\},\quad  {\overline{\Sigma}_k:=\frac{\Sigma_k-\xi_{1n}}{\delta_n}}.$$
Thus
\begin{equation}\label{Vpsi}\begin{aligned} \int\limits_{\R^4}&V_n^2\psi_n^2=k\int\limits_{\Sigma_k} V_n^2\psi_n^2= {k}\int\limits_{\Sigma_k} \Big(U_{\delta_n,\xi_{1n}} +\sum\limits_{i=2}^{k} U_{\delta_n,\xi_{in}}\Big)^2\psi_n^2\\
&=k\bigg[\int\limits_{ \overline{\Sigma}_k} U^2  \tilde \psi_n^2+\int\limits_{\Sigma_k} \Big( \sum\limits_{i=2}^{k} U_{\delta_n,\xi_{in}}\Big)^2\psi_n^2+2\int\limits_{\Sigma_k}  U_{\delta_n,\xi_{1n}} \Big(\sum\limits_{i=2}^{k} U_{\delta_n,\xi_{in}}\Big)\psi_n^2\bigg]\\
&\le k\bigg[\int\limits_{\overline{\Sigma}_k} U^2  \tilde \psi_n^2+\int\limits_{\Sigma_k} \Big( \sum\limits_{i=2}^{k} U_{\delta_n,\xi_{in}}\Big)^2\psi_n^2
+2\bigg(\int\limits_{ \overline{\Sigma}_k}  U^2 \tilde \psi_n^2\bigg)^\frac12\bigg(\int\limits_{\Sigma_k} \Big( \sum\limits_{i=2}^{k} U_{\delta_n,\xi_{in}}\Big)^2\psi_n^2\bigg)^\frac12\bigg].\end{aligned}\end{equation}
Let us choose $\vartheta>0$ small enough such that
$$\mathtt d(\xi_{in},\Sigma_k)=\rho_n\sin\frac\pi{k}\ge2\vartheta,\quad \mbox{for } i=2,\dots, k.$$
Then
$\{|x-\xi_{in}|< \vartheta\}\cap \Sigma_k=\emptyset$ and 
\begin{align*}\int\limits_{\Sigma_k}  U_{\delta_n,\xi_{in}} ^2\psi_n^2&\le \bigg(\int\limits_{\Sigma_k}   \psi_n^4\bigg)^\frac12\bigg(\int\limits_{\Sigma_k}   U_{\delta_n,\xi_{in}}^4\bigg)^\frac12\le
\|\psi_n\|_{L^4(\R^4)}^2\delta_n^2\bigg(\int\limits_{\Sigma_k}   \frac1{|x-\xi_{in}|^8}dx\bigg)^\frac12\to 0,
\end{align*}
 {Likewise, using \eqref{step2} and the decay of the function $U$, one can see that 
$$\int\limits_{\overline{\Sigma}_k} U^2  \tilde \psi_n^2\to 0,$$}
and hence, passing to the limit in \eqref{Vpsi}, 
$$\int\limits_{\R^4}3V_n^2\psi_n^2\to 0,$$ 
that implies a contradiction with \eqref{con02}. This completes the proof.
\end{proof}

 \subsection{The size of the error term $\bs{\mathcal E}$} 
Recalling the error term $\mathcal{E}=\Pi^\perp \mathcal{E}^*$ given in \eqref{errorDef}, let us denote
$$\begin{cases}
\mathcal{\bar{E}}_1:=\beta U V^2,\\
\mathcal{\bar{E}}_2:=V^3+\Delta V+\beta U^2V.
\end{cases}$$

\begin{proposition}\label{prop2} There exists $\beta_0<0$ such that for any $\beta\in[\beta_0, 0)$ and $\delta\in(0,e^{-\frac{1}{\sqrt{|\beta|}}})$, it holds
\begin{equation}\label{err1}\|\mathcal{\bar{E}}_1\|_{L^\frac43(\mathbb R^4)}+\|\mathcal{\bar{E}}_2\|_{L^{\frac43}(\R^4)}=\mathcal O(|\beta|\delta ).\end{equation}
\end{proposition}
\begin{proof}
Let $\vartheta>0$ be a small constant such that
$$|\xi_1-\xi_i|\ge {2\vartheta}\quad  i=2,\dots,k.$$
Hence
\begin{equation}\label{est01}U_{\delta, \xi_i}(x)\leq C {\delta}\quad \hbox{for any}\ x\in B(\xi_1,\vartheta),\quad i=2,\dots,k,\end{equation}
since
$$|x-\xi_i|\ge |\xi_i-\xi_1|- |x-\xi_1|\ge \frac12|\xi_i-\xi_1|\quad \hbox{where}\quad |x-\xi_1|\le \vartheta<\frac12\min\limits_{i\not=1}|\xi_i-\xi_1|.$$
Moreover (see for instance \cite[Appendix A]{MM}),
\begin{equation}\label{estA}
\sum_{j=2}^k {1\over |\xi_1 - \xi_j|^{2}} = A  k^2(1+o(1)),\quad A :=  {1\over 2\pi^{ 2}} \, \sum_{j=1}^\infty \frac1{j^ 2},
\end{equation}
Let us estimate first $\mathcal{\bar{E}}_1$. Decomposing $\mathbb R^4$ as
\begin{equation}\label{est02}\mathbb R^4=\underbrace{B(\xi_1,\vartheta)\cup\dots\cup B(\xi_k,\vartheta)}_{=:\mathtt {Int}}\cup\  \mathtt {Ext}.\end{equation}
we can write
$$\begin{aligned}\int\limits_{\mathbb R^4}|\mathcal{\bar{E}}_1|^{\frac43}&=\sum\limits_{i=1}^k\int\limits_{B(\xi_i,\vartheta)}|\mathcal{\bar{E}}_1|^{\frac43}+\int\limits_{\mathtt {Ext}}|\mathcal{\bar{E}}_1|^{\frac43}=k\int\limits_{B(\xi_1,\vartheta)}|\mathcal{\bar{E}}_1|^{\frac43}+\int\limits_{\mathtt {Ext}}|\mathcal{\bar{E}}_1|^{\frac43}.\end{aligned}$$ 
In the exterior region,  we obtain
$$\begin{aligned}\int\limits_{\mathtt {Ext}}|\mathcal{\bar{E}}_1|^{\frac43}&\leq C  {|\beta|^{\frac43}}\int\limits_{\mathtt {Ext}} \Big( \sum\limits_{i=1}^kU_{\delta, \xi_i}\Big)^\frac83 \leq C  {|\beta|^{\frac43}} k^{\frac83} \int\limits_{\mathbb R^4\setminus   B(\xi_1,\vartheta)} U_{\delta,\xi_1}^{\frac83} =\mathcal O(  {|\beta|^{\frac43}}\delta^{\frac83} ).\end{aligned}$$
In the ball $B(\xi_1,\vartheta)$ we split the error as
$$\begin{aligned}\mathcal{\bar{E}}_1&={ \beta} U\Big(U_{\delta,\xi_1}+\sum\limits_{i=2}^k U_{\delta, \xi_i}\Big)^2=\underbrace{{ \beta} U U_{\delta,\xi_1}^2}_{I_1(x)}+\underbrace{{ \beta}  U \Big(\sum\limits_{i=2}^k U_{\delta, \xi_i}\Big)^2 }_{I_2(x)} {+}\underbrace{2{ \beta} U U_{\delta,\xi_1}\Big(\sum\limits_{i=2}^k U_{\delta, \xi_i}\Big) }_{I_3(x)}, \end{aligned}$$
and then
$$ \int\limits_{B(\xi_1,\vartheta)}|I_1(x)|^{\frac43}\leq C|\beta|^\frac43\delta^{ \frac43}\int\limits_{B(0,\frac{\vartheta}{\delta })}U^\frac83=\mathcal O(|\beta|^\frac43 \delta^\frac43),$$
$$\begin{aligned}\int\limits_{B(\xi_1,\vartheta)}|I_2(x)|^{\frac43}&\leq C {|\beta|^{\frac43}}
\int\limits_{B(\xi_1,\vartheta)}  \Big(\sum\limits_{i=2}^k U_{\delta, \xi_i}\Big) ^{\frac83}=\mathcal O(  {|\beta|^{\frac43}} \delta^{\frac83}),
\end{aligned}$$
$$\begin{aligned}\int\limits_{B(\xi_1,\vartheta)}|I_3(x)|^{\frac43}& \leq C  {|\beta|^{\frac43}}\delta^{\frac{4}{3}}\delta^{\frac{8}{3}}\int\limits_{B\(0,\frac{\vartheta}{\delta }\)}\Big(\frac{1}{1+|y|^2}\Big)^{\frac{4}{3}}=\mathcal O(  {|\beta|^{\frac43}}\delta^{\frac83}).
\end{aligned}$$
Putting these estimates together we conclude that $\|\mathcal{\bar{E}}_1\|_{L^{4/3}(\R^4)}=\mathcal{O}(|\beta|\delta)$.\\
To analyze $\mathcal{\bar{ E}}_2$ notice first that
\begin{equation*}
\|\mathcal{\bar{E}}_2\|_{L^{\frac43}(\mathbb R^4)}\le  \bigg\| \underbrace{\Big(\sum\limits_{i=1}^k U_{\delta, \xi_i}\Big)^3-\sum\limits_{i=1}^k U_{\delta, \xi_i}^3}_{\mathcal{\bar{ E}}_{21}}\bigg\|_{L^{\frac43}(\mathbb R^4)}+\bigg\|\underbrace{\sum\limits_{i=1}^k{ \beta} U^2   U_{\delta, \xi_i}}_{\mathcal{\bar{ E}}_{22}} \bigg\|_{L^{\frac43}(\mathbb R^4)}.
\end{equation*}
We claim 
\begin{equation}\label{err03} 
\| \mathcal{\bar{ E}}_{21}\|_{L^{\frac43}(\mathbb R^4)}=\mathcal O(\delta^2 )\quad \hbox{and}\quad \| \mathcal{\bar{ E}}_{22}\|_{L^{\frac43}(\mathbb R^4)}=\mathcal O(|\beta|\delta ).
\end{equation}
To compute $\mathcal{\bar{ E}}_{21}$ we split as in \eqref{est02}, so that
\begin{equation*}\begin{aligned}\int\limits_{\mathbb R^4}\Big| \Big(\sum\limits_{i=1}^k U_{\delta, \xi_i}\Big)^3-\sum\limits_{i=1}^k U_{\delta, \xi_i}^3\Big|^{\frac43}&
=\sum\limits_{j=1}^k\int\limits_{B(\xi_j,\vartheta)}|\cdots|^{\frac43}+\int\limits_{\mathtt {Ext}}|\cdots|^{\frac43}=k\int\limits_{B(\xi_1,\vartheta)}|\cdots|^{\frac43}+\int\limits_{\mathtt {Ext}}|\cdots|^{\frac43}.\end{aligned}\end{equation*}
Notice that \begin{align}
\int\limits_{\mathtt {Ext}}&\Big| \Big(\sum\limits_{i=1}^k U_{\delta, \xi_i}\Big)^3-\sum\limits_{i=1}^k U_{\delta, \xi_i}^3\Big|^{\frac43} 
\leq C  \int\limits_{\mathbb R^4\setminus   B(\xi_1,\vartheta)} U_{\delta, \xi_1} ^4  =\mathcal O(\delta^4  ), \label{err05}
\end{align}
and, using \eqref{est01},
\begin{equation*}\begin{split}
\int\limits_{ B(\xi_1,\vartheta)}\Big| &\Big(\sum\limits_{i=1}^k U_{\delta, \xi_i}\Big)^3-\sum\limits_{i=1}^k U_{\delta, \xi_i}^3\Big|^{\frac43}
=\int\limits_{B(\xi_1,\vartheta)}\Big|\Big(U_{\delta,\xi_1}+\sum\limits_{i=2}^k U_{\delta, \xi_i}\Big)^3-U_{\delta,\xi_1}^3-\sum\limits_{i=2}^k U_{\delta, \xi_i}^3\Big|^{\frac43}\\
=&\int\limits_{B(\xi_1,\vartheta)}\Big|3U_{\delta,\xi_1}^2 \sum\limits_{i=2}^k U_{\delta, \xi_i}+3U_{\delta,\xi_1}\Big(\sum\limits_{i=2}^k U_{\delta, \xi_i}\Big)^2+\Big(\sum\limits_{i=2}^k U_{\delta, \xi_i}\Big)^3 -\sum\limits_{i=2}^k U_{\delta, \xi_i}^3\Big|^{\frac43}\\
\leq&\, C\Bigg(\int\limits_{ B(\xi_1,\vartheta)}\Big|U_{\delta,\xi_1}^2 \sum\limits_{i=2}^k U_{\delta, \xi_i}\Big|^{\frac43}
+\int\limits_{ B(\xi_1,\vartheta)}\Big|U_{\delta,\xi_1}\Big(\sum\limits_{i=2}^k U_{\delta, \xi_i}\Big)^2\Big|^{\frac43}\\ &+\int\limits_{ B(\xi_1,\vartheta)}\Big|\Big(\sum\limits_{i=2}^k U_{\delta, \xi_i}\Big)^3-\sum\limits_{i=2}^k U_{\delta, \xi_i}^3\Big|^{\frac43}\Bigg)\\ =& \mathcal O(\delta^{\frac83}),
\end{split}\end{equation*}
it follows $\| \mathcal{\bar{ E}}_{21}\|_{L^{\frac43}(\mathbb R^4)}=\mathcal O(\delta^2 )$.

We estimate $\mathcal{\bar{ E}}_{22}$ as
 \begin{equation}\label{err2}
 \begin{aligned}
\int\limits_{\mathbb R^4}\Big|\beta^{\frac43} \Big(\sum\limits_{i=1}^kU^2U_{\delta, \xi_i}\Big)^{\frac43} \Big|
&\leq C|\beta|^{\frac43}\delta^{\frac43} \int\limits_{\mathbb R^4}{1\over (1+|x|^2)^{\frac83}}{1\over |x-\xi_1|^{\frac83}}dx\\
&\leq C|\beta|^{\frac43}\delta^{\frac43} \bigg(\int\limits_{B(\xi_i, R)} {1\over |x-\xi_1|^{\frac83}}dx+\int\limits_{\R^4\setminus B(\xi_i, R)}{1\over (1+|x|^2)^{\frac83}}\bigg)\\
&=\mathcal O(|\beta|^{\frac43}\delta^\frac43  ).
  \end{aligned}
 \end{equation}
 Thus $\|\mathcal{\bar{E}}_{2}\|_{L^{4/3}(\R^4)}=\mathcal{O}(|\beta|\delta)$ and claim \eqref{err03} follows.
\end{proof}
The previous results allow us to apply a fixed point argument in order to solve the non-linear  system 
\begin{equation}\label{sys3proj}
\bs{\mathcal L}(\phi, \psi)=\bs{\mathcal E}+\bs{\mathcal N}(\phi, \psi)\text{~ in~} X\times K^{\perp},
\end{equation}
which corresponds to \eqref{rid2}. More precisely, we have:

\begin{proposition}\label{prop3}
There exists $\beta_0<0$ such that for any $\beta\in[\beta_0, 0)$ and $\delta\in(0,e^{-\frac{1}{\sqrt{|\beta|}}})$,   system \eqref{sys3proj} has a unique solution $(\phi,\psi)\in X\times K^{\perp}$. Furthermore, there exists a constant $C>0$ such that
\begin{equation}\label{size}\|\phi\|+\|\psi\|\leq C|\beta|\delta .\end{equation}
\end{proposition}
The proof of this result follows by Proposition \ref{prop1} and Proposition \ref{prop2} and it is standard in the literature concerning Lyapunov-Schmidt methods, so we omit it for the sake of simplicity. Let us just mention that the linear terms contained in $\bs{\mathcal{N}}$ do not cause any trouble since they all have the parameter $\beta$ in front.


\subsection{The reduced problem}\label{s4}
Consider $(\phi,\psi)$ provided by Proposition \ref{prop3}. Thus, 
$$\mathcal E_2^*+\mathcal N_2^*(\phi, \psi)-\mathcal L^*_2(\phi, \psi)=c_0 Z,$$
for some constant $c_0$ depending on $\delta$. Hence,
seeing that they also satisfy \eqref{rid1} is equivalent to find $\delta=\delta(\beta)$ such that $c_0(\delta)=0$. Testing in the equation above, we see that this constant is given by
\begin{equation}\label{coo}   c_0(\delta)=\frac{\int\limits_{\R^4}\nabla\(\mathcal E_2^*+  {\mathcal N}_2^*(\phi,\psi)-\mathcal L_2^*(\phi,\psi)\)\nabla Z\, dx}{\int\limits_{\R^4}| \nabla Z|^2\,dx}.\end{equation}
 \begin{proposition}\label{c0-esti}
There exist $\mathfrak a,\mathfrak b>0$ such that\begin{equation}\label{cod}
{c}_0(\delta)=-\mathfrak a\delta^2(1+o(1))+\mathfrak b{ \beta}\delta^2\ln(\delta )(1+o(1)),\end{equation}
where $\beta<0$ and $|\beta|$ small enough.
\end{proposition}
\begin{proof} 
Proceeding as in \eqref{tn}, we can estimate the denominator in \eqref{coo}, so let us compute the numerator. Using \eqref{size} and the definitions of $\mathcal E_2^*$ and ${\mathcal N_2^*}(\phi,\psi)$ (see \eqref{errorDef} and \eqref{quad2}), we obtain that
\begin{equation*}
\begin{aligned} \int\limits_{\R^4}&\nabla\(\mathcal E_2^*+  {\mathcal N_2}^*(\phi_1,\varphi)-{\mathcal L}_2^*(\phi_1,\varphi)\) \nabla Z dx\\
&=\underbrace{\int\limits_{\R^4}\(V^3+\Delta V\)Z dx}_{I_1}+\underbrace{{ \beta}\int\limits_{\R^4} U^2\(\sum\limits_{i=1}^k U_{\delta, \xi_i}\) Z dx}_{I_2}+o(\delta^2).
\end{aligned}
\end{equation*}
Indeed, from  \eqref{tn}, \eqref{err1}, \eqref{size} and $|Z|\leq \sum_{i=1}^kU_{\delta, \xi_i}$,
$$\begin{aligned}\int\limits_{\R^4}\nabla(\mathcal N_2^*(\phi, \psi))\nabla Zdx
&\leq C\bigg(\|\psi\|^2+|\beta|\|\phi\|^2+\int\limits_{\mathbb R^4}\beta U^2\psi Z+\int\limits_{\mathbb R^4}2\beta UV\phi Z\bigg)=\mathcal O(|\beta|^2\delta^2 ).\end{aligned}$$
By simplicity, let us denote $Z_{\delta,\xi_i}=Z_{\delta,\xi_i}^0$, with $Z_{\delta,\xi_i}^0$ defined in \eqref{z0}. Arguing as \eqref{l3}, it holds
$$\begin{aligned}\int\limits_{\R^4}-\nabla (\mathcal L_2^*(\phi, \psi))\nabla  Zdx &\leq C\|\psi\|\|\sum\limits_{i=1}^{k}U_{\delta, \xi_{i}}^2Z_{\delta, \xi_{i}}-V^2Z\|_{L^{\frac43}(\R^4)}=\mathcal O(|\beta|\delta^3).
\end{aligned}
$$

Now we claim that
 \begin{equation}\label{i1}
I_1=-\mathfrak c_1\delta^2 +o(\delta^2) \quad \hbox{for some}\ \mathfrak c_1>0,
\end{equation}
and
\begin{equation}\label{i2}I_2=   \mathfrak c_2\beta\delta^2\ln(\delta )+\mathcal O(|\beta|\delta^2) \quad \text{   for some   } \mathfrak c_2>0.\end{equation}
Let us prove first \eqref{i1}. Notice that, by symmetry,
$$I_1=k\int\limits_{\R^4}\bigg(\Big(\sum\limits_{i=1}^k U_{\delta, \xi_i}\Big)^3-\sum\limits_{i=1}^k U_{\delta, \xi_i}^3\bigg)Z_{\delta,\xi_1}dx,  $$
and thus, choosing $\vartheta$ as in \eqref{est02}, 
\begin{equation}\label{sma02}\begin{aligned}
I_1&=k\int\limits_{B(\xi_1,\vartheta)} 3U_{\delta,\xi_1}^2 \sum\limits_{i=2}^k U_{\delta, \xi_i}Z_{\delta,\xi_1}dx+\int\limits_{\R^4\setminus B(\xi_1,\vartheta)}\bigg(\Big(\sum\limits_{i=1}^k U_{\delta, \xi_i}\Big)^3-\sum\limits_{i=1}^k U_{\delta, \xi_i}^3\bigg)Z_{\delta,\xi_1}dx\\
&\quad+\int\limits_{B(\xi_1,\vartheta)} \(3U_{\delta,\xi_1}\Big(\sum\limits_{i=2}^k U_{\delta, \xi_i}\Big)^2+\Big(\sum\limits_{i=2}^k U_{\delta, \xi_i}\Big)^3 -\sum\limits_{i=2}^k U_{\delta, \xi_i}^3\)Z_{\delta,\xi_1}dx
\end{aligned}
\end{equation}
 By Taylor's expansion   for fixed $R>0$ and for any $y\in B(0, \frac{R}{\delta})$
$$\begin{aligned}& U_{\delta, \xi_i}(\delta y+\xi_1)\\
&= c_4 \frac{\delta}{|\xi_1-\xi_i|^2}\bigg[1-\frac{\delta^2+\delta^2|y|^2-2\delta\langle y, \xi_1-\xi_i\rangle}{|\xi_1-\xi_i|^2}+\mathcal O\bigg(\bigg(\frac{\delta^2+\delta^2|y|^2-2\delta\langle y, \xi_1-\xi_i\rangle}{|\xi_1-\xi_i|^2}\bigg)^2\bigg)\bigg],\end{aligned}$$
and hence
$$\begin{aligned}
k\int\limits_{B(\xi_1,\vartheta)}&3U_{\delta,\xi_1}^2\sum\limits_{i=2}^kU_{\delta, \xi_i} Z_{\delta,\xi_1}dx=k\delta\int\limits_{B(0,\frac{\vartheta}{\delta })}3U^2 {Z^0} \sum\limits_{i=2}^kU_{\delta,\xi_i}(\delta y+ \xi_1)dy\\
&=3c_4 A\delta^2k^3(1+o(1))c_4^{-1}\int\limits_{B(0,\frac{\vartheta}{\delta })}U^2(y\cdot \nabla U+U)dy\\
&= 3A\delta^2k^3(1+o(1))\bigg[-\frac13\Big(\int\limits_{\R^4}U^3dy-\int\limits_{\mathbb{R}^4\setminus B(0,\frac{\vartheta}{\delta })}U^3dy\Big)+\frac{1}{3}\int\limits_{\partial B(0,\frac{\vartheta}{\delta })}U^3\langle y, \nu\rangle dS\bigg]\\
&=-\mathfrak c_1 \delta^2+o(\delta^2 ),\end{aligned}
$$
where $\mathfrak c_1:= Ak^3\int\limits_{\mathbb R^4}U^3>0.$
Let us see that the other terms in  \eqref{sma02} are of higher order. Indeed, using \eqref{est01} and \eqref{err05}, we obtain 
$$\begin{aligned}
\bigg|\int\limits_{B(\xi_1,\vartheta)}3U_{\delta,\xi_1}\Big(\sum\limits_{i=2}^kU_{\delta, \xi_i}\Big)^2 Z_{\delta,\xi_1} dx\bigg|&=\mathcal O\bigg(\delta^4\bigg|\int\limits_{B(0,\frac{\vartheta}{\delta })}U Z^0  dy\bigg| \bigg)= \mathcal O(\delta^4|\ln(\delta )|),
\end{aligned}
$$
$$\begin{aligned}
\bigg|\int\limits_{B(\xi_1,\vartheta)}&\bigg(\Big(\sum\limits_{i=2}^k U_{\delta, \xi_i}\Big)^3 -\sum\limits_{i=2}^k U_{\delta, \xi_i}^3\bigg) Z_{\delta,\xi_1}dx\bigg|\\
&\le\bigg(\int\limits_{B(\xi_1,\vartheta)} \bigg|\Big(\sum\limits_{i=2}^k U_{\delta, \xi_i}\Big)^3 -\sum\limits_{i=2}^k U_{\delta, \xi_i}^3\bigg|^{\frac{4}{3}}dx\bigg)^{\frac{3}{4}}\|Z^0\|_{L^4(\R^4)}=\mathcal O(\delta^3),
\end{aligned}
$$and
$$\begin{aligned}
 \bigg|\int\limits_{\R^4\setminus B(\xi_1,\vartheta)}&\bigg(\Big(\sum\limits_{i=1}^k U_{\delta, \xi_i}\Big)^3 -\sum\limits_{i=1}^k U_{\delta, \xi_i}^3\bigg) Z_{\delta,\xi_1}dx\bigg|\\
&\le\bigg(\int\limits_{\R^4\setminus B(\xi_1,\vartheta)} \bigg|\Big(\sum\limits_{i=1}^k U_{\delta, \xi_i}\Big)^3 -\sum\limits_{i=1}^k U_{\delta, \xi_i}^3\bigg|^{\frac{4}{3}}dx\bigg)^{\frac{3}{4}}\|Z^0\|_{L^4(\R^4)}=\mathcal O(\delta^3).
\end{aligned}
$$
Then \eqref{i1} follows.

Let us prove \eqref{i2}.
We have
$$I_2={ \beta}k\int\limits_{\R^4} U^2U_{\delta, \xi_1} Z_{\delta,\xi_1} dx +{ \beta}\int\limits_{\R^4} U^2\sum\limits_{i\neq j} U_{\delta, \xi_i} Z_{\delta,\xi_j} dx.$$
Notice first that,  if $|x|\geq 2$, then $|x-\xi_i|\geq \frac{|x|}{2}$ , $i=1,..., k,$. Thus, 
$$\begin{aligned}
\bigg|{ \beta}\int\limits_{\R^4} U^2\sum\limits_{i\neq j} U_{\delta, \xi_i} Z_{\delta,\xi_j} dx\bigg|\leq &C|\beta| k^2\bigg[ \int\limits_{\R^4\setminus B(0, 2)}U^2  U_{\delta,  {\xi_1}} U_{\delta, {\xi_2}} dx+\int\limits_{B(0, 2)}U^2  U_{\delta,  {\xi_1}} U_{\delta, {\xi_2}} dx\bigg]\\
\leq &C|\beta|k^2\bigg[\delta^2\int\limits_{\R^4\setminus B(0, 2)}\frac{1}{|x|^8}dx+\delta^2\int\limits_{B(0, 2)}\frac{1}{|x- {\xi_1}|^2}\frac{1}{|x- {\xi_2}|^2}dx\bigg]\\
=&\mathcal O(|\beta|\delta^2).
\end{aligned}$$
Indeed, 
taking $x=| {\xi_1}- {\xi_2}|y+ {\xi_1}$, then there exists $R>0$ such that $|y|\leq R$ if $x\in B(0,2)$ since $| {\xi_1}- {\xi_2}|\geq \vartheta$, and  then
$$\begin{aligned}
\int\limits_{B(0, 2)}\frac{1}{|x- {\xi_1}|^2}\frac{1}{|x-\xi_2|^2}dx\leq &\int\limits_{B(0, R)}\frac{1}{|y|^2}\frac{1}{|y+\frac{ {\xi_1}- {\xi_2}}{| {\xi_1}- {\xi_2}|}|^2}dy\\
\leq &C\bigg[ \int\limits_{|y|\leq 2}\frac{1}{|y|^2}\frac{1}{|y+\frac{ {\xi_1}- {\xi_2}}{| {\xi_1}- {\xi_2}|}|^2}dy+\int\limits_{2\leq |y|\leq R}\frac{1}{|y|^4}dy\bigg]\\
=&\mathcal O(1).
\end{aligned} $$
Likewise
\begin{equation*}\begin{aligned}
 { \beta}k \int\limits_{\R^4} U^2 U_{\delta, \xi_1} Z_{\delta,\xi_1} dx&=\beta k\bigg[
 U^2(\xi_1)\int\limits_{B(\xi_1,\vartheta)} U_{\delta,\xi_1}Z_{\delta,\xi_1}dx +\int\limits_{B(\xi_1,\vartheta)}  (U^2-U^2(\xi_1))U_{\delta,\xi_1}Z_{\delta,\xi_1}dx  \\ 
& \quad+\int\limits_{\R^4\setminus B(\xi_1,\vartheta)}  U^2 U_{\delta, \xi_1}Z_{\delta,\xi_1}dx\bigg]\\ &= {\mathfrak c}_2\beta \delta^2\ln(\delta)+h.o.t.,
\end{aligned}\end{equation*}
where $h.o.t.$ stands for terms of higher order in $\delta$. Indeed,
$$\begin{aligned}kU^2(\xi_1)\int\limits_{B(\xi_1,\vartheta)} & U_{\delta,\xi_1}Z_{\delta,\xi_1}dx=k\(\frac12c_4+o(1)\)^2 \delta^2 \int\limits_{B(0,\frac{\vartheta}{\delta })}U(y)\Psi_0(y)dy\\
 &=k\(\frac14c_4^2+o(1)\)\delta^2c_4^{-1} \int\limits_{ B(0,\frac{\vartheta}{\delta })}U(y)(y\cdot \nabla U(y)+U(y))dy\\
 &=k\(\frac14c_4^2+o(1)\)\left(-\frac{\delta^2}{c_4}\int\limits_{B(0,\frac{\vartheta}{\delta })}U^2dy+\frac{\delta^2}{2c_4} \int\limits_{\partial B(0,\frac{\vartheta}{\delta })}U^2\langle y, \nu\rangle dy\right)\\
&={\mathfrak c}_2\delta^2(\ln(\delta)+\mathcal O(1)),\end{aligned}$$
where   $\mathfrak c_2:= \frac14c_4^3k>0.$
Furthermore, by \eqref{err2},
$$\begin{aligned} \bigg|\int\limits_{B(\xi_1,\vartheta)} (U^2(x)-U^2(\xi_1))U_{\delta,\xi_1}Z_{\delta,\xi_1}dx\bigg|&\le C \int\limits_{B(\xi_1,\vartheta)}
|x-\xi_1|{\delta^2\over\(\delta^2+|x-\xi_1|^2\)^2}dx =\mathcal O(\delta^2),\end{aligned}$$
 and  
\begin{equation*}
\begin{aligned}
\bigg|\int\limits_{\R^4\setminus B(\xi_1,\vartheta)} U^2 U_{\delta, \xi_1}Z_{\delta,\xi_1}dx\bigg|&\le C\Big\|U^2\sum\limits_{i=1}^k U_{\delta, \xi_i}\Big\|_{L^{\frac43}(\R^4)}\bigg(\int\limits_{\R^4\setminus B(\xi_1,\vartheta)}Z_{\delta,\xi_1}^4dx\bigg)^{\frac14}\\
&=\mathcal O(\delta^2).
\end{aligned}
\end{equation*}
Then \eqref{i2} follows.
\end{proof}
We can now complete the proof of Theorem \ref{thm01}.
\begin{proof}[Proof of Theorem \ref{thm01}] By Proposition \ref{prop3} and Proposition \ref{c0-esti}, it is enough to fix the parameter $\delta\in\(0,e^{-\frac{1}{\sqrt{|\beta|}}}\)$
such that ${c}_0(\delta)=0$, given in \eqref{cod}. In fact, we can choose
\begin{equation*}
\delta ={e^{-d_\beta}}\quad  \hbox{with}\quad  d_\beta=\frac1{|\beta|}\frac{\mathfrak a}{\mathfrak b}+o(1)>0,
\end{equation*}
so that 
\begin{equation*}-\mathfrak a-\mathfrak b{ \beta}  d_\beta+  o\(1\)=-\mathfrak a+\mathfrak b{ \beta}  \ln(\delta)+  o\(1\)=0.\end{equation*}
\end{proof}

\section{The general case $m\geq 3$}\label{sec4}
Consider the general system with $m=q+1$ components, $q\in \mathbb{N}$, 
\begin{equation}\label{1}-\Delta u_i= u_i^3+ \beta_{ij}\sum\limits_{j=1\atop i\not=j}^{q+1} u_i u_j^2\quad \hbox{in}\ \mathbb R^4, \quad i=1,...,q+1,\end{equation}
where $\beta_{ij}$ are coupling constants. Inspired by Section \ref{sec2}, we would like to find an approximation that will allow us to transform \eqref{1} into a system of two components, to perform a Lyapunov-Schmidt reduction there. Using \cite{dPMPP2} as starting point, we denote
$$ {\mathscr T}_\theta:=\(\begin{matrix}\cos\theta&-\sin\theta&0&0\\
\sin\theta&\cos\theta&0&0\\
0&0&\cos\theta&\sin\theta\\
0&0&-\sin\theta&\cos\theta \\
 \end{matrix}\).$$
Notice that
$${\mathscr T}_0=I,\quad {\mathscr T}_{\theta_1}\circ {\mathscr T}_{\theta_2}={\mathscr T}_{\theta_1+\theta_2},\quad {\mathscr T}_{\theta}^{-1}={\mathscr T}_{-\theta},\quad {\mathscr T}_{\theta+\pi}=-{\mathscr T}_\theta,\quad \text{~for~}\theta\in[0, 2\pi],$$
and we set
$${\mathscr T}_i:={\mathscr T}_{\frac{(i-1)\pi}q},\quad i=1,\ldots, q.$$
Assume that the constants $\beta_{ij}$ in \eqref{1} satisfy \eqref{beta} with $m=q+1$, i.e.
$$\beta_{ij} =\alpha\quad \hbox{and}\quad
\beta_{i(q+1)}=\beta_{(q+1)j} =\beta\ \hbox{if}\ i,j=1,\ldots, q,$$
for certain $\alpha,\beta\in \R$, with $\beta<0$, and consider the system 
\begin{equation}\label{2222}\left\{\begin{aligned}
&-\Delta u=u^3+\beta u  \sum\limits_{r=1}^qv_r^2\quad \hbox{in}\ \mathbb R^4,\\
&-\Delta v=v^3+\beta vu^2+\alpha v \sum\limits_{r=2}^qv_r^2\quad \hbox{in}\ \mathbb R^4,
\end{aligned}\right.
\end{equation}
where 
\begin{align*} 
&v_r(x):=v({\mathscr T}_r x)\ \hbox{for any}\ r=1,\dots,q.
\end{align*}
Thus, if $(u,v)$ solves \eqref{2222} and satisfies 
\begin{align}\label{42}
  &u(x)=u({\mathscr T}_rx)\ \hbox{for any}\  r=2,\dots,q,\\
&v(x)=v(-x),\label{41}
 \end{align}
 then $(u_1,\dots, u_{q+1})$, with 
\begin{align*}
&u_i(x)=v({\mathscr T}_i x)\ \hbox{for any}\ i=1,\dots,q\quad
 \hbox{and}\quad u_{q+1}(x)=u(x),\label{5}
 \end{align*}
solves \eqref{1}. Indeed, for every $i=1,\dots,q$,
$$\begin{aligned}-\Delta u_i(x)&=-\Delta v({\mathscr T}_i x)=\Big(v^3+\beta vu^2+\alpha v \sum\limits_{r=2}^qv_r^2\Big)({\mathscr T}_i x)\\
&=u_i^3(x)+\beta u_i(x)u^2({\mathscr T}_ix)+\alpha u_i(x) \sum\limits_{r=2}^qv_r^2({\mathscr T}_ix)\end{aligned}$$
 and
 $$\begin{aligned}\sum\limits_{r=2}^qv_r^2({\mathscr T}_ix)&=\sum\limits_{r=2}^qv^2({\mathscr T}_r\circ {\mathscr T}_i\ x)=\sum\limits_{r=2}^qv^2\Big({\mathscr T}_{\frac{(r+i-2)\pi}q}x\Big)=\sum\limits_{j=i+1}^{q+i-1}v^2\Big({\mathscr T}_{\frac{(j-1)\pi}q}x\Big)\\ 
 & =\sum\limits_{j=i+1}^{q}\underbrace{v^2\Big({\mathscr T}_{\frac{(j-1)\pi}q}x\Big)}_{=u_j^2(x)}+\sum\limits_{j=q+1}^{q+i-1}v^2\Big({\mathscr T}_{\frac{(j-1)\pi}q}x\Big)\\
 &=\sum\limits_{j=i+1}^{q} u_j^2(x) +\sum\limits_{\ell=1}^{i-1}\underbrace{v^2\Big({\mathscr T}_{\frac{(\ell+q-1)\pi}q}x\Big)}_{=v ^2\Big(-{\mathscr T}_{\frac{(\ell -1)\pi}q}x\Big)=u_\ell^2(x)} =\sum\limits_{j=1\atop j\not=i}^qu_j^2(x).\end{aligned}$$
Thus, if we find a solution $(u,v)$ to \eqref{2222} satisfying \eqref{42}, \eqref{41}, we will have obtained a solution to the general system \eqref{1}. We will search for them in the form
\begin{equation*}u=U+\phi,\qquad v=\underbrace{\sum\limits_{\ell=1}^kU_{\delta,  \tilde\xi_\ell}}_{\tilde V}+\psi,\ \end{equation*}
 where   
\begin{equation*}  \tilde\xi_\ell:= \frac{\rho}{\sqrt2}\Big( \cos\frac{2\pi(\ell-1)} k, \sin\frac{2\pi(\ell-1)} k, \cos\frac{2\pi(\ell-1)} k, \sin\frac{2\pi(\ell-1)} k\Big),\quad \ell=1,\dots,k,
\end{equation*}
$U_{\delta, \tilde{\xi}_\ell}$ is given in \eqref{bubble}, and $\delta,\rho>0$ such that
\begin{equation*} \delta^2+\rho^2=1.\end{equation*}
Then, for every $r=2,\dots,q$,
$$v_r=\underbrace{\sum\limits_{\ell=1}^kU_{\delta, \tilde \xi_\ell^r}}_{\tilde V_r}+\psi_r,\quad  U_{\delta,\tilde\xi_\ell^r}(x):=U_{\delta,\tilde\xi_\ell}({\mathscr T}_rx),\qquad  \psi_r(x):=\psi(\mathscr T_rx),$$
and $$\begin{aligned}\tilde \xi_\ell^r:={\mathscr T}_r^{-1}\tilde\xi_\ell=\frac{\rho}{\sqrt2} &\Big(\cos\Big(-\frac{(r-1)\pi}q+ \frac{2\pi(\ell-1)} k\Big) ,\sin\Big(-\frac{(r-1)\pi}q+ \frac{2\pi(\ell-1)} k\Big),\\ &\cos\Big(\frac{(r-1)\pi}q+ \frac{2\pi(\ell-1)} k\Big) ,\sin\Big(\frac{(r-1)\pi}q+ \frac{2\pi(\ell-1)} k\Big)\Big).\end{aligned}$$

\begin{remark}
If $k$ is an even integer then $\tilde V$  satisfies \eqref{41} and $U$ satisfies \eqref{42}, since it is radially symmetric. Thus we will need $\phi$ and $\psi$ to satisfy \eqref{42} and \eqref{41} respectively.
\end{remark}

\begin{remark}
There exists a positive constant $c$ such that
\begin{equation}\label{dist}|\tilde\xi_i^r-\tilde\xi^s_j|^2=2\Big(1-\cos\frac{(r-s)\pi}q\cos\frac{(i-j)2\pi}k\Big)\geq c>0 \ \hbox{for any}\ r\neq s.\end{equation}
Similarly to \eqref{estA}, it also holds\begin{equation}\label{estB}
\sum_{j=2}^k {1\over |\tilde\xi_1 - \tilde\xi_j|^{2}} = A  k^2(1+o(1)),\quad \text{with } A\text{~is a positive constant}.
\end{equation}
\end{remark}

\begin{remark}
It is not necessary to assume $\alpha<0$.
\end{remark}
Let us identify the invariances that we will need for the functional setting. Given $f\in \mathcal{D}^{1,2}(\R^4)$, let us consider:
\smallskip

$\bullet$ Evenness in both $x_2$ and $x_4$ coordinates, i.e., 
\begin{equation}\label{24}
f(x_1,x_2,x_3,x_4)=f(x_1,-x_2,x_3,-x_4).
\end{equation}

$\bullet$ Invariance under interchanging the $(x_1,x_2)$ and $(x_3,x_4)$ coordinates, i.e.,
\begin{align}f(x_1,x_2,x_3,x_4) =f(x_3,x_4,x_1,x_2).\label{e611}\end{align}

$\bullet$ Invariance under the linear tranformation $\mathscr R_k$,
  \begin{align} f(x)=f(\mathscr R_k   x),\quad 
 \label{e62}
  \mathscr R_k:= \(\begin{matrix}
 \cos\frac{2\pi}k& -\sin\frac{2\pi}k&0&0\\
 \sin\frac{2\pi}k&\cos\frac{2\pi}k&0&0\\
0&0&\cos\frac{2\pi}k&- \sin\frac{2\pi}k\\
0& 0&\sin\frac{2\pi}k& \cos\frac{2\pi}k\\
 \end{matrix}\).\end{align}
 
 $\bullet$ Invariance under the transformation ${\mathscr T}_r$, i.e.,
\begin{equation}\label{e63}f(x)=f({\mathscr T}_r x)\quad \hbox{for any}\ r=2,\dots,q.\end{equation}
Notice that, for even $k$, applying \eqref{e62} a total of $k/2$ times (that is, doing a rotation of $\pi$-angle) we obtain
$$f(x_1,x_2,x_3,x_4)= f(-x_1,-x_2,-x_3,-x_4).
$$

Due to the special symmetries needed to reduce the system to two equations, we will need $\phi$ and $\psi$ to belong to different spaces. We will search for them respectively in the spaces
\begin{equation*}\begin{split}
\tilde{X}_1&:=\{f\in \mathcal D^{1,2} (\mathbb R^4): \hbox{ $f$ satisfies}\ \eqref{24}, \eqref{e611}, \eqref{e62}, \eqref{e63} \hbox{ and } \eqref{kel}\},\\
\tilde{X}_2&:=\{f\in \mathcal D^{1,2} (\mathbb R^4): \hbox{  $f$ satisfies}\ \eqref{24},  \eqref{e611}, \eqref{e62} \hbox{ and } \eqref{kel}\}.
\end{split}\end{equation*}
We define
$$\tilde{ K}:=\tilde X_2\cap \mathtt{span} \{\tilde Z\},\quad \tilde Z:=\sum\limits_{i=1}^kZ_{\delta,\tilde\xi_i}^0,\quad \tilde{K}^\perp:=\left\{\psi\in \tilde X_2:\, \langle\psi,\tilde Z\rangle=0\right\},$$
and the orthogonal projections 
$$\tilde{\Pi}:\tilde X_1\times \tilde X_2\to \tilde X_1\times \tilde{K},\quad \tilde{\Pi}^\perp:\tilde X_1\times \tilde X_2\to \tilde X_1\times \tilde{K}^\perp,$$
where $Z^0_{\delta, \tilde\xi_i}$ is defined according to \eqref{z0}. Thus the system \eqref{2} can be rewritten as 
\begin{equation}\label{genpi} \tilde\Pi \{\bs{\tilde{\mathcal L}}(\phi, \psi)-\bs{\tilde{\mathcal E}}-\bs{\tilde{\mathcal N}}(\phi, \psi)\}=0,
\end{equation}
\begin{equation}\label{sysgen}
\tilde\Pi^{\perp}\{\bs{\tilde{\mathcal L}}(\phi, \psi)-\bs{\tilde{\mathcal E}}-\bs{\tilde{\mathcal N}}(\phi, \psi)\}=0,
\end{equation}
where the linear operator $\bs{\tilde{\mathcal L}}=(\tilde{\mathcal L_1},\tilde{\mathcal L_2})$ is defined by
\begin{equation}\label{lineargen}\left\{\begin{aligned}
&\tilde{\mathcal L_1}(\phi, \psi):= \phi-(-\Delta)^{-1}(3U^2\phi)\quad \text{~in~}\R^4,\\
&\tilde{\mathcal L_2}(\phi, \psi):=  \psi-(-\Delta)^{-1}(3\tilde V^2\psi )\quad \text{~in~}\R^4,
\end{aligned}\right.\end{equation}
the  error term $\bs{\tilde{\mathcal E}}=(\tilde{\mathcal E_1},\tilde{\mathcal E_2})$ by
$$\left\{\begin{aligned}
&\tilde{\mathcal E_1}:=(-\Delta)^{-1}(\beta U\sum_{r=1}^q\tilde V_r^2)\quad \text{~in~}\R^4,\\
&\tilde{\mathcal E_2}:=(-\Delta)^{-1}(\Delta \tilde V+\tilde V^3+\beta U^2 \tilde V+\alpha \tilde V\sum_{r=2}^q\tilde V_r^2)\quad \text{~in~} \R^4, 
 \end{aligned}\right.$$
and the non-linear term $\bs{\tilde{\mathcal N}}(\phi, \psi)=(\tilde{\mathcal N_1}(\phi, \psi),\tilde{\mathcal N_2}(\phi, \psi))$ is given by
$$\left\{\begin{aligned}\tilde{\mathcal N_1}(\phi, \psi):=&\,(-\Delta)^{-1}[\phi^3+3U\phi^2+\beta(2\phi\sum_{r=1}^q\tilde V_r\psi_r+\phi\sum_{r=1}^q\psi_r^2+U\sum_{r=1}^q\psi_r^2)\\&+\beta \sum_{r=1}^q\tilde V_r^2\phi+2\beta U\sum_{r=1}^q\tilde V_r\psi_r]\text{~in~}\R^4, \\
\tilde{\mathcal N_2}(\phi, \psi):=&\,(-\Delta)^{-1}[\psi^3+3\tilde V\psi^2+\beta(\psi\phi^2+2U\psi\phi+\tilde V\phi^2)+\alpha(2\psi\sum_{r=2}^q\tilde V_r\psi_r+\psi\sum_{r=2}^q\psi_r^2
\\&+\tilde V\sum_{r=2}^q\psi_r^2)+\beta (U^2\psi+2 U\tilde V\phi)+\alpha(\sum_{r=2}^q\tilde V_r^2\psi+2\tilde V\sum_{r=2}^q\tilde V_r\psi_r)]\ \text{~in~} \R^4.\end{aligned}\right.$$
We can prove a linear solvability result in the spirit of Proposition \ref{prop1}.
\begin{proposition}\label{prop4.1}
Assume $k\geq 2$ even.  There exist $c>0$ and  $\beta_0<0$ such that for each $\beta \in [\beta_0, 0)$, $\delta\in\big(0,e^{-\frac{1}{\sqrt{|\beta|}}}\big),$  
\begin{equation}\label{apriori33}\|\tilde{\Pi}^{\perp}\bs{\tilde{\mathcal L}}(\phi,\psi)\|\ge c\|(\phi,\psi)\|, \quad \forall\, (\phi, \psi)\in \tilde X_1\times \tilde{K}^{\perp}.\end{equation}
where $\bs{\tilde{\mathcal L}}$ is defined as in \eqref{lineargen}. In particular, the inverse operator $(\tilde{\Pi}^{\perp}\bs{\tilde{\mathcal L}})^{-1}: \tilde X_1\times \tilde{K}^\perp\to \tilde X_1\times \tilde{K}^\perp$ exists and is continuous.
\end{proposition}
\begin{proof}
The scheme of the proof is similar to Proposition \ref{prop1}, so we will only write in detail the differences, mainly related to the orthogonalities with the elements of the kernel. 

Let us start proving \eqref{apriori33} by contradiction. Suppose that there exist $\beta_n\to 0$, $\phi_n\in \tilde X_1, \psi_{n} \in    \tilde{X}_2$ satisfying 
$$\tilde{\Pi}^{\perp}\bs{\tilde{\mathcal L}}(\phi_{n}, \psi_{n})=(h_{1n}, h_{2n}) \in \tilde X_1\times \tilde{K}^\perp$$
and
$$\| \phi_{n}\|+\|\psi_{n}\|=1, \quad  \|{ h_{1n}}\| +\|{ h_{2n}}\| \to 0\mbox{ as }n\to +\infty.$$
\textbf{Step 1:} Proceeding as before, one can see that $\phi_n\rightharpoonup \phi^*$ weakly in $\mathcal{D}^{1,2}(\R^4)$. To prove that indeed $\phi^*=0$ we see that
$$\tilde I_{in}:=\int\limits_{\mathbb R^4}U^2Z^i\phi_{n}(y)dy=0, \quad i=0,1,...,4,$$
with $Z^i$ given in \eqref{lin02}. Indeed, using \eqref{24},
$$\tilde I_{2n}=\int\limits_{\mathbb R^4} c_4^2{y_2\over(1+|y|^2)^4}\phi_n(y)dy=0=\int\limits_{\mathbb R^4} c_4^2{y_4\over(1+|y|^2)^4}\phi_n(y)dy=\tilde I_{4n}.$$
Moreover, by \eqref{e62},
\begin{align*}\tilde I_{1n}&=\int\limits_{\mathbb R^4} c_4^2{y_1\over(1+|y|^2)^4}\phi_n(y)dy =\int\limits_{\mathbb R^4} c_4^2{\cos{2\pi\over k}z_1\over(1+|z|^2)^4} \phi_n(\mathscr R_{k}(z))dz\\
&= \cos{2\pi\over k}\int\limits_{\mathbb R^4} U^2(z) \phi_n(z)Z^1(z)dz=\cos{2\pi\over k} \tilde I_{1n},
\end{align*}
and, since $k\ge2$, necessarilly  $\tilde I_{1n}=0$. Analogously it can be seen that $\tilde I_{3n}=0$, and $\tilde I_{0n}$ follows by \eqref{kel}. With this we can conclude that $\phi^*=0$.

\noindent\textbf{Step 2:} Setting 
$$\tilde\psi _{n}(y):=\delta_n\psi_n(\delta_n y+\tilde\xi_{1n })$$
where $\tilde\xi_{1n}:=\frac{\rho_n}{\sqrt{2}}(1, 0, 1, 0),$ and proceeding as in Proposition \ref{prop1}
it can be seen that (up to a subsequence) 
\begin{equation*}\tilde\psi_n\rightharpoonup\tilde\psi\ \hbox{weakly in}\ \mathcal D^{1,2}(\mathbb R^4)\ \hbox{and strongly in}\ L^p_{loc}(\mathbb R^4)\ \hbox{for any}\ p\in [2,4),\end{equation*}
as $n\to+\infty.$ The goal is to prove that actually $\tilde\psi=0$, what will follow if we prove 
\begin{equation*}\tilde J_{in}:=\int\limits_{ \R^4}   \tilde{\psi}_n U ^2 Z^idx=0,\quad i=1,...,4,\end{equation*}
since $\psi_n\in \tilde K^{\perp}$. By definition $\tilde\psi_n$ inherits the symmetry \eqref{24}, and thus $\tilde J_{2n}=\tilde J_{4n}=0$.
For the cases $i=1,3$ we will use the Kelvin invariance \eqref{kel}.  Indeed, proceeding as in \eqref{Kelv13} it follows
 \begin{equation*}\begin{aligned}\int\limits_{  \R^4} \tilde \psi_n(y){y_1\over(1+|y|^2)^4}dy=& \int\limits_{  \R^4}  \tilde\psi_n(y){  y_1 \over(1+|y|^2)^4}dy +\frac{\rho_n}{ \delta_n}\int\limits_{\R^4} \tilde \psi_n(y){ 1-|\delta_n y+\tilde\xi_{1n}|^2\over(1+|y|^2)^4}dy,
\end{aligned}
\end{equation*}
and then using that $\psi_n\in \tilde K^\perp$,
$$\begin{aligned}0&=\int\limits_{\R^4} \tilde  \psi_n(y){ 1-|\delta_n y+\tilde\xi_{1n}|^2\over(1+|y|^2)^4}dy\\
&= \delta_n^2 \int\limits_{\R^4}  \tilde \psi_n(y) {1-  |y|^2\over(1+|y|^2)^4}dy-2\delta_n\rho_n \int\limits_{ \R^4}  \tilde \psi_n(y){  y_1 +y_3\over(1+|y|^2)^4}dy\\
&=-2\delta_n\rho_n \int\limits_{ \R^4} \tilde  \psi_n(y){  y_1 +y_3\over(1+|y|^2)^4}dy.\end{aligned}$$
Furthermore, since $\psi_n$ satisfies \eqref{e611}, we deduce that
$$\int\limits_{ \R^4} \tilde  \psi_n(y){  y_1\over(1+|y|^2)^4}dy=\int\limits_{ \R^4} \tilde  \psi_n(y){ y_3\over(1+|y|^2)^4}dy,$$
and therefore necessarily $\tilde J_{1n}=\tilde J_{3n}=0,$
which ends this step.

\noindent \textbf{Step 3:}  The delicate point is to prove 
$$\int_{\R^4}\tilde V_n^2\psi_n^2\to 0\quad \mbox{ as }n\to\infty.$$
We consider the cone $\tilde\Sigma_k$ defined as
\begin{equation*}\tilde \Sigma_k:=\left\{(r_1\cos\theta_1,r_1\sin\theta_1, r_2\cos\theta_2,r_2\sin\theta_2)\ :\ (\theta_1,\theta_2)\in\Lambda_k,\ r_1,r_2\geqslant 0\right\},\end{equation*}
with $\Lambda_k:=\Lambda^1_k\cup\Lambda^2_k\cup\Lambda_k^3\subset[-\pi,\pi]\times[-\pi,\pi]$, 
\begin{equation*}\begin{aligned}\Lambda_k^1:=&\bigg\{\left[-\frac\pi k,\frac\pi k\right]\cup\left[\pi-\frac\pi k,\pi+\frac\pi k\right]\bigg\}\times\bigg\{\left[-\frac\pi k,\frac\pi k\right]\cup\left[\pi-\frac\pi k,\pi+\frac\pi k\right]\bigg\},\\
\Lambda^2_k:=&\bigg\{\left[-\frac\pi k,\frac\pi k\right]\cup\left[\pi-\frac\pi k,\pi+\frac\pi k\right]\bigg\}\times \bigg\{\left[-\pi+\frac\pi k,-\frac\pi k\right]\cup\left[\frac\pi k,\pi-\frac\pi k\right]\bigg\},\\
\Lambda^3_k:=& \bigg\{\left[-\pi+\frac\pi k,-\frac\pi k\right]\cup\left[\frac\pi k,\pi-\frac\pi k\right]\bigg\} \times\bigg\{\left[-\frac\pi k,\frac\pi k\right]\cup\left[\pi-\frac\pi k,\pi+\frac\pi k\right]\bigg\}.
\end{aligned}\end{equation*}
 It could be useful to visualize the symmetry via the following picture for
$k=6$. Here  
$$\Lambda_k^1=\mbox{
 \begin{tikzpicture}[samples=100, domain=0:0.5]
\draw[help lines, color=black, step=0.5cm ]
(0,0) grid (0.5,0.5);
\filldraw[pattern=grid]
(0,0)--(0,0.5)--(0.5,0.5)--(0.5,0);
\end{tikzpicture}}
 \qquad \Lambda_k^2=\mbox{
\begin{tikzpicture}[samples=100, domain=0:0.5]
\draw[help lines, color=black, step=0.5cm ]
(0,0) grid (0.5,0.5);
\filldraw[pattern=vertical lines]
(0,0)--(0,0.5)--(0.5,0.5)--(0.5,0);
\end{tikzpicture}}
\qquad \Lambda_k^3=\mbox{
\begin{tikzpicture}[samples=100, domain=0:0.5]
\draw[help lines, color=black, step=0.5cm ]
(0,0) grid (0.5,0.5);
\filldraw[pattern=horizontal lines]
(0,0)--(0,0.5)--(0.5,0.5)--(0.5,0);
\end{tikzpicture}}.$$
Any function   invariant with respect to the isometries \eqref{24}, \eqref{e611}, \eqref{e62} and \eqref{e63}    takes the same value in the squares  having the  same colors.

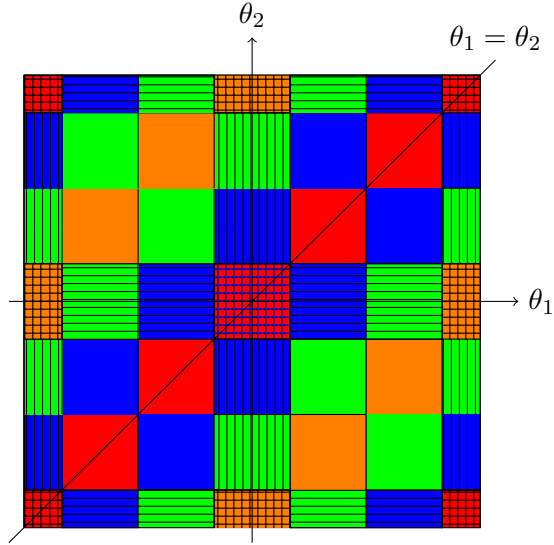
\begin{center}
\begin{tikzpicture}[samples=500,domain=-3:3]
\draw[help lines, color=black, step=0.5cm ] 
(-3, -3)  grid (3, 3);
 
\filldraw[fill= red]  
(-0.5,-0.5)--(-0.5, 0.5)--  (0.5,0.5)--(0.5,-0.5)  ;
\filldraw[fill= red]  
(-0.5,-0.5)--(-0.5, -1.5)--  (-1.5,-1.5)--(-1.5,-0.5)  ;
\filldraw[fill= red]  
(-1.5,-1.5)--(-1.5, -2.5)--  (-2.5,-2.5)--(-2.5,-1.5)  ;
 \filldraw[fill= red]  
(0.5,0.5)--(0.5, 1.5)--  (1.5,1.5)--(1.5,0.5)  ;
\filldraw[fill= red]  
(1.5,1.5)--(1.5, 2.5)--  (2.5,2.5)--(2.5,1.5)  ;
\filldraw[fill= red]  
(2.5,2.5)--(2.5, 3)--  (3,3)--(3,2.5)  ;
\filldraw[fill= red]  
(-2.5,-2.5)--(-2.5, -3)--  (-3,-3)--(-3,-2.5)  ;

\filldraw[fill= blue]  
(-0.5,-0.5)--(-0.5, -1.5)-- (0.5,-1.5) -- (0.5,-0.5) ;
\filldraw[fill= blue]  
(0.5,0.5)--(0.5, -0.5)-- (1.5,-0.5) -- (1.5,0.5) ;
\filldraw[fill= blue]  
(1.5,1.5)--(1.5, 0.5)-- (2.5,0.5) -- (2.5,1.5) ;
\filldraw[fill= blue]  
(2.5,2.5)--(2.5, 1.5)-- (3,1.5) -- (3,2.5) ;
\filldraw[fill= blue]  
(-1.5,-1.5)--(-1.5, -2.5)-- (-0.5,-2.5) -- (-0.5,-1.5) ;
\filldraw[fill= blue]  
(-2.5,-2.5)--(-2.5, -3)-- (-1.5,-3) -- (-1.5,-2.5) ;

\filldraw[fill= blue]  
(-0.5,0.5)--(-0.5, 1.5)--  (0.5,1.5)--(0.5,0.5)  ;
\filldraw[fill= blue]  
(0.5,1.5)--(0.5, 2.5)--  (1.5,2.5)--(1.5,1.5)  ;
\filldraw[fill= blue]  
(1.5,2.5)--(1.5, 3)--  (2.5,3)--(2.5,2.5)  ;
\filldraw[fill= blue]  
(-1.5,-0.5)--(-1.5, 0.5)--  (-0.5,0.5)--(-0.5,-0.5)  ;
\filldraw[fill= blue]  
(-2.5,-1.5)--(-2.5, -0.5)--  (-1.5,-0.5)--(-1.5,-1.5)  ;
\filldraw[fill= blue]  
(-3,-2.5)--(-3, -1.5)--  (-2.5,-1.5)--(-2.5,-2.5)  ;

\filldraw[fill= green]  
(-0.5,1.5)--(-0.5, 2.5)--  (0.5,2.5)--(0.5,1.5)  ;
\filldraw[fill= green]  
(0.5,2.5)--(0.5, 3)--  (1.5,3)--(1.5,2.5)  ;
\filldraw[fill= green]  
(-1.5,0.5)--(-1.5, 1.5)--  (-0.5,1.5)--(-0.5,0.5)  ;
\filldraw[fill= green]  
(-2.5,-0.5)--(-2.5, 0.5)--  (-1.5,0.5)--(-1.5,-0.5)  ;
\filldraw[fill= green]  
(-3,-1.5)--(-3, -0.5)--  (-2.5,-0.5)--(-2.5,-1.5)  ;
\filldraw[fill= green]  
(-0.5,-2.5)--(-0.5, -1.5)--  (0.5,-1.5)--(0.5,-2.5)  ;
\filldraw[fill= green]  
(-1.5,-3)--(-1.5, -2.5)--  (-0.5,-2.5)--(-0.5,-3)  ;
\filldraw[fill= green]  
(0.5,-1.5)--(0.5, -0.5)--  (1.5,-0.5)--(1.5,-1.5)  ;
\filldraw[fill= green]  
(1.5,-0.5)--(1.5, 0.5)--  (2.5,0.5)--(2.5,-0.5)  ;
\filldraw[fill= green]  
(2.5,0.5)--(2.5, 1.5)--  (3,1.5)--(3,0.5)  ;

\filldraw[fill= orange]  
(0.5,-2.5)--(0.5, -1.5)--  (1.5,-1.5)--(1.5,-2.5)  ;
\filldraw[fill= orange]  
(-0.5,-3)--(-0.5, -2.5)--  (0.5,-2.5)--(0.5,-3)  ;
\filldraw[fill= orange]  
(1.5,-1.5)--(1.5, -0.5)--  (2.5,-0.5)--(2.5,-1.5)  ;
\filldraw[fill= orange]  
(2.5,-0.5)--(2.5, 0.5)--  (3,0.5)--(3,-0.5)  ;
\filldraw[fill= orange]  
(-0.5,2.5)--(-0.5, 3)--  (0.5,3)--(0.5,2.5)  ;
\filldraw[fill= orange]  
(-1.5,1.5)--(-1.5, 2.5)--  (-0.5,2.5)--(-0.5,1.5)  ;
\filldraw[fill= orange]  
(-2.5,0.5)--(-2.5, 1.5)--  (-1.5,1.5)--(-1.5,0.5)  ;
\filldraw[fill= orange]  
(-3,-0.5)--(-3, 0.5)--  (-2.5,0.5)--(-2.5,-0.5)  ;

\filldraw[fill= green]  
(-2.5,1.5)--(-2.5, 2.5)--  (-1.5,2.5)--(-1.5,1.5)  ;
\filldraw[fill= green]  
(-3,0.5)--(-3, 1.5)--  (-2.5,1.5)--(-2.5,0.5)  ;
\filldraw[fill= green]  
(-1.5,2.5)--(-1.5, 3)--  (-0.5,3)--(-0.5,2.5)  ;
\filldraw[fill= green]  
(2.5,-1.5)--(2.5, -2.5)--  (1.5,-2.5)--(1.5,-1.5)  ;
\filldraw[fill= green]  
(3,-0.5)--(3, -1.5)--  (2.5,-1.5)--(2.5,-0.5)  ;
\filldraw[fill= green]  
(1.5,-2.5)--(1.5, -3)--  (0.5,-3)--(0.5,-2.5)  ;

\filldraw[fill= red]  
(-2.5,2.5)--(-2.5, 3)--  (-3,3)--(-3,2.5)  ;
\filldraw[fill= red]  
(2.5,-2.5)--(2.5, -3)--  (3,-3)--(3,-2.5)  ;

\filldraw[fill= blue]  
(1.5,-2.5)--(1.5, -3)--  (2.5,-3)--(2.5,-2.5)  ;
\filldraw[fill= blue]  
(2.5,-1.5)--(2.5, -2.5)--  (3,-2.5)--(3,-1.5)  ;

\filldraw[fill= blue]  
(-1.5,2.5)--(-1.5, 3)--  (-2.5,3)--(-2.5,2.5)  ;
\filldraw[fill= blue]  
(-2.5,1.5)--(-2.5, 2.5)--  (-3,2.5)--(-3,1.5)  ;

\begin{scope}
\filldraw[pattern=grid ]
(-0.5,-0.5)--(-0.5, 0.5)--(0.5, 0.5)--(0.5, -0.5);
\filldraw[pattern=grid ]
(-3,-0.5)--(-3, 0.5)--(-2.5, 0.5)--(-2.5, -0.5);
\filldraw[pattern=grid ]
(3,-0.5)--(3, 0.5)--(2.5, 0.5)--(2.5, -0.5);
\filldraw[pattern=grid ]
(-0.5,-3)--(-0.5, -2.5)--(0.5, -2.5)--(0.5, -3);
\filldraw[pattern=grid ]
(-0.5,3)--(-0.5, 2.5)--(0.5, 2.5)--(0.5, 3);
\filldraw[pattern=grid ]
(-3,2.5)--(-3, 3)--(-2.5, 3)--(-2.5, 2.5);
\filldraw[pattern=grid ]
(2.5,2.5)--(2.5, 3)--(3, 3)--(3, 2.5);
\filldraw[pattern=grid ]
(-3,-3)--(-3, -2.5)--(-2.5, -2.5)--(-2.5, -3);
\filldraw[pattern=grid ]
(2.5,-3)--(2.5, -2.5)--(3, -2.5)--(3, -3);

\filldraw[pattern=horizontal lines ]
(-2.5,2.5)--(-2.5,3)--(-0.5,3)--(-0.5,2.5);
\filldraw[pattern=horizontal lines ]
(0.5,2.5)--(0.5,3)--(2.5,3)--(2.5,2.5);
\filldraw[pattern=horizontal lines ]
(-2.5,-0.5)--(-2.5,0.5)--(-0.5,0.5)--(-0.5,-0.5);
\filldraw[pattern=horizontal lines ]
(0.5,-0.5)--(0.5,0.5)--(2.5,0.5)--(2.5,-0.5);
\filldraw[pattern=horizontal lines ]
(-2.5,-2.5)--(-2.5,-3)--(-0.5,-3)--(-0.5,-2.5);
\filldraw[pattern=horizontal lines ]
(0.5,-2.5)--(0.5,-3)--(2.5,-3)--(2.5,-2.5);

\filldraw[pattern=vertical lines  ]
(2.5,-2.5)--(3,-2.5)--(3,-0.5)--(2.5,-0.5);
\filldraw[pattern=vertical lines ]
(2.5,0.5)--(3,0.5)--(3,2.5)--(2.5,2.5);
\filldraw[pattern=vertical lines ]
(-0.5,-2.5)--(0.5,-2.5)--(0.5,-0.5)--(-0.5,-0.5);
\filldraw[pattern=vertical lines ]
(-0.5,0.5)--(0.5,0.5)--(0.5,2.5)--(-0.5,2.5);
\filldraw[pattern=vertical lines ]
(-2.5,-2.5)--(-3,-2.5)--(-3,-0.5)--(-2.5,-0.5);
\filldraw[pattern=vertical lines ]
(-2.5,0.5)--(-3,0.5)--(-3,2.5)--(-2.5,2.5);

\draw[pattern=grid, pattern color=blue] (-0.5,-0.5) rectangle (0.5,0.5);

\end{scope}

%
\draw[->] (-3.2,0) -- (3.5,0) node[right] {$\theta_1$};
\draw[->] (0,-3.2) -- (0,3.5) node[above] {$\theta_2$};
\draw[-] (-3.2,-3.2) -- (3.2,3.2) node[above] {$\theta_1=\theta_2$};

\end{tikzpicture}
\end{center}
 {Given a function $f\in L^1(\mathbb R^4)$  satisfying \eqref{e62} we have
$$\frac{k}{2}\bigg(2\int\limits_{\Lambda_k^1}f\,dx+\int\limits_{\Lambda_k^2}f\,dx+\int\limits_{\Lambda_k^3}f\,dx\bigg)=2\int\limits_{\R^4}f\,dx,$$
and then there exists a constant $C>0$ such that
$$\int_{\R^4}\tilde V_n^2\psi_n^2\leq Ck\int_{\tilde\Sigma_k}\tilde V_n^2\psi_n^2.$$
The contradiction arises analogously to Proposition \ref{prop1}.}
\end{proof}
To analyze the size of the error we follow the ideas in Proposition \ref{prop2}.
\begin{proposition}\label{errormc}
There exists $\beta_0<0$ such that, for every $\beta\in [\beta_0,0)$ and $\delta\in (0,e^{-\frac{1}{\sqrt{|\beta|}}})$, it holds
\begin{equation*}\|(-\Delta)\tilde{\mathcal{E}}_1\|_{L^\frac43(\mathbb R^4)}+
\|(-\Delta)\tilde{\mathcal E}_2\|_{L^\frac43(\R^4)}=\mathcal O(|\beta|\delta).
\end{equation*}
\end{proposition}
\begin{proof} 
 {Due to \eqref{dist}, it is easy to see that, proceeding as in \eqref{err03}, one has that 
$$\|\Delta \tilde V+\tilde V^3+\beta U^2\tilde V\|_{L^{\frac 43}(\R^4)}=\mathcal{O}(|\beta|\delta),$$}
so we only need to handle the new terms
$$\|\beta U\sum_{r=1}^q\tilde V_r^2\|_{L^\frac43(\mathbb R^4)}\quad\mbox{ and }\quad  \|\alpha \sum_{r=2}^q\tilde V_r^2 \tilde V\|_{L^\frac43(\R^4)}.$$
Since $\tilde V_r(x)=\tilde V(\mathscr T_r x)$, we have
$$\bigg(\int\limits_{\R^4}|  \beta U\sum_{r=1}^q\tilde V_r^2|^{\frac43}dx\bigg)^{\frac34}\leq C|\beta| \|U\tilde V^2\|_{L^{\frac43}(\R^4)}=\mathcal O(|\beta|\delta).$$
Let $\vartheta>0$ small enough. Using \eqref{dist} we obtain that
\begin{equation}\label{sigk12}
|\tilde V_r(x)|\leq C\sum_{i=1}^{{k}}\frac{\delta}{|x-\tilde\xi^r_i|^2}\leq C{ \delta }\quad x\in B( {\tilde\xi_j}, \vartheta), \quad r=2,...,q, \quad {j=1,..., k}.
\end{equation}
Consequently,  we obtain
\begin{equation}\label{term1}\begin{split}
\int\limits_{B( {\tilde\xi_j}, \vartheta)}|\tilde V_r^2\tilde V|^{\frac 43}&\leq C { \delta^{\frac 83}}\int\limits_{B( {\tilde\xi_j}, \vartheta)}U_{\delta, {\tilde\xi_j}}^{\frac 43}= \mathcal O(\delta^{4}),\quad  {j=1,..., k}.
\end{split}\end{equation}
Likewise,
\begin{equation}\label{sigk22}
|\tilde V(x)|\leq C\sum_{i=1}^k\frac{\delta}{|x-\tilde\xi_i|^2}\leq C\delta,\quad x\in B(\tilde\xi^r_j, \vartheta), \quad r=2,..., q, \quad j=1,...,k,
\end{equation}
and hence
\begin{equation}\label{term2}\begin{split}
\int\limits_{B(\tilde\xi^r_j, \vartheta)}|\tilde V_r^2\tilde V|^{\frac 43}&\leq C\delta^{\frac 43} \int\limits_{B(\tilde\xi^r_j, \vartheta)}\big(\sum_{i=1}^kU_{\delta,\tilde\xi^r_j}\big)^{\frac 83}=\mathcal O(\delta^{\frac 83}).
\end{split}\end{equation}
Finally, 
\begin{equation}\label{term3} 
\begin{aligned}
\int\limits_{\R^4\setminus (\cup_{r=1}^q\cup_{i=1}^kB(\tilde\xi^r_i, \vartheta))}\tilde V_r^{\frac{8}{3}}\tilde V^{\frac43}&\leq C\delta^{\frac43} \int\limits_{\R^4\setminus B( {\tilde\xi^2_1},{ \vartheta})}U_{\delta, {\tilde\xi^2_1}}^{\frac83}=\mathcal O(\delta^{4} ).
\end{aligned}
\end{equation}
Putting \eqref{term1}, \eqref{term2} and \eqref{term3} together we conclude that
$$\|\alpha \sum\limits_{r=2}^q\tilde V_r^2\tilde V\|_{L^{\frac 43}(\R^4)}=\mathcal O(\delta^2),$$
and the result follows.
\end{proof}
By Proposition \ref{prop4.1} and Proposition \ref{errormc}, a fixed point argument will allow us to solve the non-linear  system 
\begin{equation}\label{sysmul}
\tilde{\Pi}^{\perp}\bs{\tilde{\mathcal L}}(\phi, \psi)=\tilde{\Pi}^{\perp}[\bs{\tilde{\mathcal E}}+\bs{\tilde{\mathcal N}}(\phi, \psi)]\text{~ in~} \tilde X_1\times \tilde{K}^{\perp},
\end{equation}
which corresponds to \eqref{sysgen}.
\begin{proposition}\label{prop3.5}
Assume $k\geq 2$ even. There exists $\beta_0<0$ such that for any $\beta\in[\beta_0, 0)$ and $\delta\in(0,e^{-\frac{1}{\sqrt{|\beta|}}})$,   system \eqref{sysmul} has a unique solution $(\phi,\psi)\in \tilde X_1\times \tilde{K}^{\perp}$.  Furthermore, there exists a constant $C>0$ such that
\begin{equation}\label{apriorim3}
\|\phi\|+\|\psi\|\leq C|\beta|\delta .
\end{equation}
\end{proposition}
\begin{proof}
It is sufficient to prove that  $$\tilde{\Pi}^{\perp}[\bs{\tilde{\mathcal E}}+\bs{\tilde{\mathcal N}}(\phi, \psi)] \in \tilde X_1\times \tilde{K}^{\perp}\quad \mbox{for every pair}\quad (\phi,\psi)\in \tilde X_1\times \tilde{K}^{\perp},$$
 which holds if and only if
$\tilde{\mathcal E_1},\, \tilde{\mathcal N_1}\in \tilde X_1, ~ \tilde{\mathcal E_2},\, \tilde{\mathcal N_2} \in \tilde X_2.$ Hence, we need to check carefully that all the symmetries in the definitions of the spaces are satisfied. Notice that
$$\tilde V_r(x)=\tilde V({\mathscr T}_rx)=\sum\limits_{\ell=1}^kU_{\delta,  {\tilde\xi_\ell}}({\mathscr T}_rx)=\sum\limits_{\ell=1}^kU_{\delta,   {\tilde\xi^r_\ell}}(x)$$
satisfies
$$\tilde V_r(  \mathscr R_kx)=\tilde V({\mathscr T}_r\circ \mathscr R_kx)=\tilde V({\mathscr T}_r x)=\tilde V_r(x).$$
Then, for even $k$, \begin{equation*}\tilde V_r(-x)=\tilde V_r(\mathscr R_k^{\frac{k}{2}}x)=\tilde V_r(x),\end{equation*}
and, as a consequence,
$$\begin{aligned}
\sum\limits_{r=2}^q\tilde V_r^2(x_3, x_4, x_1, x_2)=&\sum\limits_{r=2}^q\tilde V_r^2((\mathscr T_r^2)^{-1}x)=\sum\limits_{r=2}^q\tilde V^2(\mathscr T_r\circ(\mathscr T_r^2)^{-1}x)\\=&\sum\limits_{r=2}^q\tilde V^2(\mathscr T_r^{-1}x)
=\sum\limits_{r=2}^q\tilde V^2(-\mathscr T_r^{-1}x)\\
=&\sum\limits_{r=2}^q\tilde V^2(\mathscr T_{q+1}\circ\mathscr T_r^{-1}x)
=\sum\limits_{r=2}^q\tilde V_r^2(x_1, x_2, x_3, x_4).\end{aligned}$$
Analogously $\psi_r$ satisfies  \eqref{e62}, and then $\sum\limits_{r=2}^q\psi_r^2$ and $\sum\limits_{r=2}^q\tilde V_r\psi_r$ satisfy \eqref{e611}. 

 {Let us check that $\sum\limits_{r=2}^q\tilde V_r^2$ satisfies \eqref{24}. Indeed, since $\tilde{V}$ satisfies \eqref{24} and \eqref{e611},
\begin{align*}\sum\limits_{r=2}^q\tilde V_r^2&(x_1, x_2, x_3, x_4)=\sum\limits_{r=2}^q\tilde V^2(\mathscr T_r(x_1, x_2, x_3, x_4)  )\\
=&\sum\limits_{r=2}^q\tilde V^2\left(x_1\cos\frac{(r-1)\pi}{q}-x_2\sin\frac{(r-1)\pi}{q}, x_1\sin\frac{(r-1)\pi}{q}+x_2\cos\frac{(r-1)\pi}{q}, \right.\\&\quad \quad \quad \quad\left.x_3\cos\frac{(r-1)\pi}{q}+x_4\sin\frac{(r-1)\pi}{q}, x_3\sin\frac{(r-1)\pi}{q}-x_4\cos\frac{(r-1)\pi}{q}\right)\\
=&\sum\limits_{r=2}^q\tilde V^2\left(x_1\cos\frac{(r-1)\pi}{q}-x_2\sin\frac{(r-1)\pi}{q}, -x_1\sin\frac{(r-1)\pi}{q}-x_2\cos\frac{(r-1)\pi}{q}, \right.\\&\quad \quad \quad \quad\left.x_3\cos\frac{(r-1)\pi}{q}+x_4\sin\frac{(r-1)\pi}{q}, -x_3\sin\frac{(r-1)\pi}{q}+x_4\cos\frac{(r-1)\pi}{q}\right)\\
=&\sum\limits_{r=2}^q\tilde V^2\left(x_3\cos\frac{(r-1)\pi}{q}+x_4\sin\frac{(r-1)\pi}{q}, -x_3\sin\frac{(r-1)\pi}{q}+x_4\cos\frac{(r-1)\pi}{q},\right.\\
&\quad \quad \quad \quad \left.\cos\frac{(r-1)\pi}{q}x_1-\sin\frac{(r-1)\pi}{q}x_2, -\sin\frac{(r-1)\pi}{q}x_1-\cos\frac{(r-1)\pi}{q}x_2\right)\\
=&\sum\limits_{r=2}^q\tilde V_r(x_3, -x_4, x_1, -x_2)
=\sum\limits_{r=2}^q\tilde V_r(x_1, -x_2, x_3, -x_4),
\end{align*}
and similarly $\sum\limits_{r=2}^q\psi_r^2$ and $\sum\limits_{r=2}^q\tilde V_r\psi_r$ satisfy \eqref{24}. }

Applying the Banach Fixed Point Theorem can conclude the existence of a solution $(\phi,\psi)$ to \eqref{sysmul} satisfying \eqref{apriorim3}.
\end{proof}
To see that $(\phi,\psi)$ provided by Proposition \ref{prop3.5} also satisfy \eqref{genpi} we will find $\delta=\delta(\beta)$ such that 
\begin{equation*}   \tilde{c}_0(\delta):=\frac{\int\limits_{\R^4}\nabla\(\tilde{\mathcal E}_2+\tilde{\mathcal N}_2(\phi, \psi)-\tilde{\mathcal L}_2(\phi, \psi)\)\cdot \nabla \tilde Z dx}{\int\limits_{\R^4}| \nabla \tilde Z|^2dx}=0,\end{equation*}
in the spirit of Proposition \ref{c0-esti}.
Thanks to \eqref{dist}, \eqref{estB}, Proposition \ref{errormc} and \ref{prop3.5}, repeating the computions in Proposition \ref{c0-esti}, it is immediate to derive the following result:
 \begin{proposition}\label{c3-esti}
There exist $\tilde{\mathfrak a},\tilde{\mathfrak b}>0$ such that\begin{equation*}
\tilde{c}_0(\delta)=- {\tilde{\mathfrak a}}\delta^2(1+o(1))+ {\tilde{\mathfrak b}}{ \beta}\delta^2\ln(\delta )(1+o(1)),\end{equation*}
where $\beta<0$ and $|\beta|$ small enough.
\end{proposition}
\begin{proof} 
Using the definitions of ${\tilde{\mathcal E}}_2$ and  $\tilde{\mathcal{N}}_2(\phi,\psi)$, \eqref{dist}, \eqref{estB}, Proposition \ref{errormc} and Proposition \ref{prop3.5}, it can be checked that
\begin{align}
\int\limits_{\R^4}\nabla(\tilde{\mathcal E}_2+\tilde{\mathcal N}_2(\phi, \psi)-\tilde{\mathcal L}_2(\phi, \psi))\nabla \tilde Z
=&\int\limits_{\R^4}\nabla{\tilde{\mathcal E}_2}\nabla \tilde Z+o(\delta^2)\nonumber\\
=&-{\mathfrak c}_1\delta^2+{\mathfrak c}_2{ \beta}\delta^2\ln(\delta )+o(\delta^2 ).\nonumber\end{align}
where ${\mathfrak c}_1:=Ak^3\int\limits_{\mathbb R^4}U^3>0$, and ${\mathfrak c}_2= \frac14c_4^3k>0$ are the same as in \eqref{i1} and \eqref{i2}. Indeed, given the computations already performed in the proof of Proposition \ref{c0-esti}, to estimate the projection of the error term it is enough to compute 
\begin{align*}\int\limits_{\R^4}\alpha\sum\limits_{r=2}^q \tilde V_r^2 \tilde V\tilde Z .\end{align*}
Fix a small constant $\vartheta>0$. Using  \eqref{sigk12}  and noticing that $|\tilde Z|\leq C\tilde V$, it is simple to check that
\begin{equation*}
\begin{aligned}
\bigg|\int\limits_{B(\tilde\xi_i, \vartheta)}\tilde V_r^2 \tilde V\tilde Zdx\bigg|\leq&C\delta^2\int\limits_{B(0, \frac{\vartheta}{\delta })}\delta^2\frac{|1-|y|^2|}{(1+|y|^2)^3}dy=\mathcal O(\delta^4(1+|\ln(\delta )|), \quad i=1,...,k,
\end{aligned}
\end{equation*}
Likewise, by \eqref{sigk22}, 
\begin{equation*}
\begin{aligned}
\bigg|\int\limits_{B(\tilde\xi^r_j,\vartheta)}\tilde V_r^2 \tilde V\tilde Z dx\bigg|\leq&C\delta^2\int\limits_{B(\tilde\xi^r_j,\vartheta)} U_{\delta,\tilde\xi^r_j}^2 dx
\leq C\delta^4\int\limits_{B(0,\frac{\vartheta}{\delta })}\frac{1}{(1+|y|^2)^2}dy=\mathcal O(\delta^4  |\ln(\delta )|),
\end{aligned}
\end{equation*}
for every $r=2,...q$, $j=1,...,k$ and, using \eqref{term3}, it holds
\begin{equation*}
\begin{aligned}
&\bigg|\int\limits_{\R^4\setminus (\cup_{r=1}^q\cup_{i=1}^kB(\tilde\xi^r_i, \vartheta))}\sum\limits_{r=2}^q\tilde V_r^2 \tilde V\tilde Z dx\bigg|\leq C \bigg(~\int\limits_{\R^4\setminus (\cup_{i=1}^kB(\tilde\xi^r_i, \vartheta) ) }\tilde V_r^{\frac{8}{3}}\tilde V^{\frac43}\bigg)^{\frac34}\| \tilde Z\|_{L^4(\R^4)}
= \mathcal O(\delta^3).
\end{aligned}
\end{equation*}
Therefore, $$\int\limits_{\R^4}\alpha \sum\limits_{r=2}^q\tilde V_r^2 \tilde V\tilde Z dx= \mathcal O(\delta^3 ),$$
and the result follows.
\end{proof}

\begin{proof}[Proof of Theorem \ref{thm1.3}] We can also choose 
$$
\delta ={e^{-\tilde{d}_{\beta}}}
,\quad \mbox{ with }\quad \tilde{d}_\beta=\frac1{|\beta|}\frac{ {\tilde{\mathfrak a}}}{ {\tilde{\mathfrak b}}} +o(1)>0,
$$
such that $\tilde{c}_0(\delta) =0$. Thus, as a consequence of Proposition \ref{c3-esti}, the pair of solutions to \eqref{sysgen} provided by Proposition \ref{prop3.5} also solves \eqref{genpi}. This concludes the proof of the general case.
\end{proof}
\begin{remark}
In the particular case of three components, $m=3$, another construction can be made, starting from the ideas in \cite{MMW}. Indeed, consider $m=3$ in \eqref{1}, and assume 
\begin{equation}\label{betas}
{\beta_{13}=\beta_{23}= \beta_{31}=\beta_{32}=\beta<0,\quad\beta_{12}=\beta_{21}=\alpha.}
\end{equation}
Consider the system
\begin{equation}\label{2}\left\{\begin{aligned}
&-\Delta u=u^3+\beta u  v^2+\beta u{\tilde v}^2\quad \hbox{in}\ \mathbb R^4,\\
&-\Delta v=v^3+\beta vu^2+\alpha v {\tilde v}^2\quad \hbox{in}\ \mathbb R^4,\end{aligned}\right.
\end{equation}
where
\begin{equation*}\tilde v(x):=v(Tx) \quad \hbox{and}\quad T(x_1,x_2,x_3,x_4):=(x_3,x_4,x_1,x_2).
\end{equation*}
If $ u$ and $v$ solve  \eqref{2} with the symmetry
\begin{equation*}
u(x)=u(T x),
\end{equation*}
then the functions  $u_1:=v,$ $u_2(x):=u_1(Tx)=\tilde v$ and $u_3=u$
solve  \eqref{1} whenever \eqref{betas} holds. Hence it is enough with finding solutions to \eqref{2}.  This can be done considering the same approximation as in the case of $m=2$, but imposing different symmetries in the functional spaces. Namely, let us take
\begin{equation}\label{ans3}
u=U+\phi,\quad v=\underbrace{\sum\limits_{i=1}^k U_{\delta,\xi_i}}_{V}+\psi,\end{equation}
where $U_{\delta, \xi_i}$ and $\xi_i$,  $i=1,...,k$, are defined in \eqref{bubble} and \eqref{bas03}.
Consider the invariances
 \begin{equation}\label{sim12} \psi(x_1,x_2,x_3,x_4)= \psi(x_3,x_4,x_1, x_2)=\psi(Tx),\end{equation}
  \begin{equation}\label{sim10} \psi(x_1,x_2,x_3,x_4)=\psi(x_1,-x_2,x_3,x_4),\end{equation}
 \begin{equation}\label{sim11} \psi(x_1,x_2,x_3,x_4)= \psi(x_1,x_2,x_3,-x_4),\end{equation}
\begin{equation}\label{sim20} \psi( x_1,x_2,\Theta_k( x_3,x_4))=\psi(x_1,x_2,x_3,x_4),
 \end{equation}
 with $\Theta_k$ defined in \eqref{sim2}, and the associated spaces
\begin{equation*} \begin{split}
\hat{X}_1&:=\{\phi\in \mathcal D^{1,2}(\mathbb R^4): \ \phi\ \hbox{satisfies \eqref{sim12}, \eqref{sim10}, \eqref{sim2} and \eqref{kel}}\},\\
\hat{X}_2&:=\{\psi\in \mathcal D^{1,2}(\mathbb R^4): \ \psi\ \hbox{satisfies \eqref{sim10}, \eqref{sim11}, \eqref{sim20},  \eqref{sim2} and \eqref{kel}}\}.
\end{split}\end{equation*}
We ask the remainder terms $\phi$  and $\psi$ in \eqref{ans3} to
  belong to the spaces $\hat{X}_1$ and $\hat{X}_2$ respectively. Proceeding as in Section \ref{sec4}, one can find solutions to \eqref{1}, $m=3$, with the form
$$ u_1=V+\psi_1, \quad u_2= \hat V+\psi_2, \quad u_3=U+\phi$$
where $$\hat V(x):=V(Tx)=\sum\limits_{i=1}^k U_{\delta,\eta_i}(x),\quad  \psi_2(x)=\psi_1(Tx)=\psi(Tx),$$
and
$$\eta_i:=\rho\(0,0,\cos{2\pi(i-1)\over k},\sin{2\pi(i-1)\over k}\).$$
This construction is particular for the case of three components and cannot be extended to the general case.
\end{remark}

\end{document}